\documentclass[12pt]{amsart}

\usepackage{mathtools}
\def\mbinom#1#2{\ensuremath{\left(\kern-.3em\left(\genfrac{}{}{0pt}{}{#1}{#2}\right)\kern-.3em\right)}}

\usepackage[color,matrix,arrow, all, arc]{xy}
\usepackage{graphicx}
\usepackage{color}
\usepackage{caption}
\usepackage{etoolbox}
\usepackage[margin=1in]{geometry}
\usepackage{amsmath,amsthm,amssymb,tikz,tikz-cd}
\usepackage{stmaryrd}
\usepackage{aliascnt}
\usepackage{xcolor}
\definecolor{darkspringgreen}{rgb}{0.14, 0.7, 0.3}
\definecolor{melon}{rgb}{0.1, 0.5, 1}
\usepackage[colorlinks, citecolor = darkspringgreen, linkcolor = melon]{hyperref}
\usepackage{cleveref}
\usepackage{thmtools}
\usepackage{comment}
\usepackage{tikz}
\usepackage[normalem]{ulem}
\usetikzlibrary{positioning,calc,arrows.meta,shapes.geometric,fit,decorations.markings} %decorations.markings lets us draw arrows on circles, used in fig: ice fork coloring

\def\AB#1{(\textcolor{blue}{#1})}

\newcommand{\cS}[1]{{\noindent\textsf{\color{purple}$\blacksquare$~#1~$\blacksquare$}}} 

\newtheorem{thm}{Theorem}[section]
\newtheorem{cor}[thm]{Corollary}
\newtheorem{prop}[thm]{Proposition}
\newtheorem{lem}[thm]{Lemma}
\newtheorem{conj}[thm]{Conjecture}

\theoremstyle{definition}
\newtheorem{defn}[thm]{Definition}

\newtheorem{exmp}[thm]{Example}

\newtheorem{definition}[thm]{Definition}

\newtheorem{theorem}[thm]{Theorem} %[section]
\newtheorem{lemma}[thm]{Lemma}

\theoremstyle{remark}
\newtheorem{rem}[thm]{Remark}

\newtheorem{obs}[thm]{Observation}

\makeatletter
\let\c@equation\c@thm
\makeatother
\numberwithin{equation}{section}

\makeatletter
\def\@tocline#1#2#3#4#5#6#7{\relax
  \ifnum #1>\c@tocdepth % then omit
  \else
    \par \addpenalty\@secpenalty\addvspace{#2}%
    \begingroup \hyphenpenalty\@M
    \@ifempty{#4}{%
      \@tempdima\csname r@tocindent\number#1\endcsname\relax
    }{%
      \@tempdima#4\relax
    }%
    \parindent\z@ \leftskip#3\relax \advance\leftskip\@tempdima\relax
    \rightskip\@pnumwidth plus4em \parfillskip-\@pnumwidth
    #5\leavevmode\hskip-\@tempdima
      \ifcase #1
       \or\or \hskip 1em \or \hskip 2em \else \hskip 3em \fi%
      #6\nobreak\relax
    \hfill\hbox to\@pnumwidth{\@tocpagenum{#7}}\par% <---- \dotfill -> \hfill
    \nobreak
    \endgroup
  \fi}
\makeatother

\newcommand{\N} { \mathbb{N}}

\newcommand{\Z} { \mathbb{Z}}

\DeclareMathOperator{\sign}{sign}

\newcommand{\mut} {\text{mut}}
\newcommand{\cut}{\setminus}
\makeatletter
\def\subsection{\@startsection{subsection}{3}%
  \z@{.5\linespacing\@plus.7\linespacing}{.1\linespacing}%
  {\bfseries}}
\makeatother

\title{Eventual Sign Coherence}
\author{Amanda Burcroff}
\address{\hspace{-.3in} Department of Mathematics, MIT,
Cambridge, MA 02139, USA}
\email{\href{mailto:amandabu@mit.edu}{amandabu@mit.edu}}
\thanks{}
\author{Scott Neville}
\address{Laboratoire d'Algèbre, de Combinatoire et d’informatique Mathématique (LACIM)\\ Université du Québec à Montréal\\ Montréal, Québec, Canada }
\email{\href{mailto:nevilles@umich.edu}{nevilles@umich.edu}} 
\thanks{\emph{Funding.} The first author was supported by the NSF GRFP, NSF grant DMS 2503411, the Jack Kent Cooke Foundation, the UC PPFP, and the MIT School of Science.
The second author was supported by a predoctoral fellowship from the University of Michigan, the NSF Grant No. DMS-1929284, while he was in residence at ICERM in Providence during the Categorification and Computation in Algebraic Combinatorics semester program, and DMS-1840234, while a graduate student, as well as the NSERC Discovery Grant RGPIN-2022-03960 and the Canada Research Chairs program, grant number CRC-2021-00120.}


\begin{document}
\begin{abstract} The sign coherence of $c$-vectors is one of the fundamental theorems of cluster algebras with principal coefficients.  In 2019, Gekhtman and Nakanishi posed the \emph{asymptotic sign coherence conjecture} for arbitrary cluster algebras of geometric type, which says sign coherence should eventually hold in any sufficiently generic infinite mutation sequence.  
We prove that their conjecture holds almost always for skew-symmetric cluster algebras of arbitrary rank.  That is, we prove that with probability $1$, the sequence of $c$-vectors obtained by random mutation of an arbitrary  quiver eventually becomes sign-coherent. 
Our results also establish the conjecture in full generality for many families of quivers by studying a new class of \emph{brog} quivers.
\end{abstract}

\maketitle

%\tableofcontents

\section{Introduction}

\emph{Cluster algebras} are commutative rings that admit a combinatorial framework for constructing distinguished generators with many desirable properties.  A dizzying variety of fields have benefited from the study of cluster algebras, including mathematical physics, differential geometry, representation theory, Lie theory, integrable systems, and Teichmuller theory.  In the most common setting, these algebras are essentially determined by the combinatorial data of a \emph{quiver} or, equivalently, a skew-symmetric matrix.  A quiver is a directed graph with no oriented $1$ or $2$-cycles on which we can perform \emph{mutations}, which are elementary transformations of the quiver associated to each vertex.  
This mutation combinatorics is surprisingly rich and leads to many deep properties of cluster algebras, such such as the Laurent phenomenon, Laurent positivity, the finite-type classification, the existence of nice bases, and the sign coherence phenomenon.  The main goal of our paper is proving that sign coherence is much more prevalent than expected.
%As one example, if your initial quiver is acyclic (as a directed graph), then every mutation-equivalent acyclic quiver will differ only by an isomorphism after reorienting some of the arrows.

The remarkable sign coherence phenomenon is often stated purely combinatorially, but the only known proofs rely on heavy machinery from representation theory and algebraic geometry. Partition the vertices of the quiver into sets of \emph{mutable vertices} (whose cardinality $n$ is the \emph{rank} of the quiver) and \emph{frozen} vertices, and forbid mutations at the frozen vertices. Each mutable vertex $i$ in a quiver has an associated \emph{$c$-vector}, whose entries record the arrows between $i$ and each frozen vertex $u$: $(c_i)_u = \# (i \rightarrow u) - \# (u \rightarrow i)$.
A quiver is \emph{principally framed} if the $c$-vectors are precisely the standard basis vectors of $\mathbb{R}^n$, which yields \emph{principal coefficients} for the corresponding cluster algebra. The \emph{sign coherence phenomenon} states after any sequence of mutations on a principally framed quiver, the new $c$-vectors will be \emph{sign-coherent}, meaning that for each $i$, the entries of $c_i$ will either be non-negative or non-positive.  It was originally conjectured in 2007 by Fomin and Zelevinsky \cite{FZiv} and was not fully resolved until 2018 \cite{GHKK}.

The sign coherence of $c$-vectors has shown up in surprising ways in the study of cluster algebras.  It is equivalent to the $F$-polynomial of a cluster algebra with principal coefficients having constant term $1$.  These $F$-polynomials were given a representation-theoretic interpretation in the more general setting of quivers with potential by Derksen-Weyman-Zelevinsky \cite{DWZ2}, and this connection enabled them to first prove sign coherence for principally-framed quivers.  Plamondon \cite{Plamondon} and Nagao \cite{Nagao} later gave proofs using cluster categories and Donaldson-Thomas theory, respectively. Sign coherence later arose in the groundbreaking work of Gross-Hacking-Keel-Kontsevich \cite{GHKK} on cluster scattering diagrams, which has greatly informed our understanding of Laurent positivity and canonical bases for cluster algebras.  They gave a geometric proof of sign coherence which holds in the more general skew-symmetrizable setting. Nakanishi and Zelevinsky \cite{NZ} showed that the $c$-vectors, which parametrize the coefficients, are tropically dual to the \emph{$g$-vectors}, which parametrize the cluster variables.  Their proof of this duality, which relies on the sign coherence of $c$-vectors, solidified the connection to the representation theory of finite-dimensional algebras, where $g$-fans play an important role.  Cao, Huang, and Li  \cite{CaoHuangLi} proved that seeds are uniquely determined by their $C$-matrix (whose columns are the $c$-vectors) using both sign coherence and Laurent positivity. 

%In other words, each mutable vertex $i$ is a either a source or a sink with respect to all the frozen vertices. 

Since the sign coherence phenomenon requires the initial quiver to be principally framed and hence sign-coherent, at face value it seems like a statement about a property being preserved under mutation.  Our results show that not only is sign coherence preserved under the right circumstances, but it emerges eventually after almost any sequence of mutations.  This behavior was first identified by Gekhtman and Nakanishi \cite[{cf. \Cref{{conj: asymptotic sign coherence}}}]{GekhtmanNakanishi}, who posed the \emph{asymptotic sign coherence conjecture} saying that all sufficiently generic mutation sequences of quivers (with arbitrary initial $c$-vectors) will eventually (or asymptotically) be sign-coherent. Via a direct computation of $c$-vectors, they confirmed their conjecture in rank $2$ and for a particular mutation sequence on the rank $3$ Markov quiver. Our main result substantially addresses their conjecture for quivers of arbitrary rank in an entirely combinatorial way, and gives new tools to understand the sign coherence phenomenon. 

\begin{thm}[{cf. \Cref{thm: geq 3 eventually sc}}]
\label{thm: intro main}
Let $Q$ be a quiver whose mutable part is mutation-infinite.  For all $i \in \N$, let $m_i$ be chosen uniformly at random from $[n]$.  With probability $1$, the mutation sequence $\mathbf M = m_1m_2\cdots$ on $Q$ is eventually sign-coherent.
\end{thm}

 We prove this in two steps. We first identify a class of quivers called \emph{ice forks} (based on the \emph{forks} introduced by Warkentin \cite{Warkentin}), and show that ice forks are only a short mutation sequence away from being sign-coherent. Sign coherence is also preserved for ice forks under any reduced sequence of mutations, so after the initial few mutations, the quiver will remain sign-coherent. 
The second step involves showing that the above situation occurs with probability $1$ in a random (balanced, monotone) sequence of mutations.   These same methods yield a combinatorial proof that for a principally-framed quiver with mutation-infinite mutable part, almost all of the quivers in its mutation graph are sign-coherent (\Cref{cor:principal framing result}).  Our approach is similar to the methods of Warkentin \cite[Chapter 5] {Warkentin} for quivers, but adapted to accommodate frozen vertices.  We focus specifically on the setting where the mutable part is \emph{mutation-infinite}, meaning that it is mutation-equivalent to infinitely many distinct quivers. This includes almost all quivers, with the mutation-finite quivers being classified in \cite{FST}. 

We additionally prove that Gekhtman--Nakanishi's asymptotic sign coherence holds in full for \emph{mutation-abundant} rank $3$ quivers (\Cref{thm: mutation cyclic eventually sign coherent}).  Mutation-abundant quivers include all mutation classes of acyclic quivers with at least $2$ arrows between each pair of vertices, and all classes which do not contain an acyclic quiver. Our methods involve introducing and studying \emph{brog} quivers, which have vertices colored with $4$ colors respecting some prescribed arrow orientations. Brog quivers generalize several properties of ice forks, see in particular \Cref{rem: ice forks and coloring}.
We show that any sufficiently nice mutation results in a brog quiver (\Cref{thm: brog}) and that brog quivers are eventually sign-coherent on \emph{cycle-preserving} mutation sequences that mutate at each mutable vertex (\Cref{thm: brog coherence}).  Our combinatorial methods also recover Gekhtman and Nakanishi's results for rank $2$ quivers and the Markov quiver.

\textbf{Additional related work.}
Of the many works on sign coherence, ours is among a short list that addresses the phenomenon primarily via combinatorial methods.  Another work that takes this combinatorial approach is Ervin \cite{ErvinRS}, showing that if a fork quiver (Definition~\ref{def:fork}) is sign-coherent and satisfies some mild conditions, then any mutation sequence starting at this quiver will only visit sign-coherent quivers. 
Grant \cite{Grant2026} recently built on this result to show that certain $c$-vectors stabilize after sufficiently many mutations.
Some other related works have combinatorial aspects but primarily use algebraic or geometric methods. Machacek and Ovenhouse \cite{MO} give an alternate proof of asymptotic sign coherence for rank two cluster algebras in the context of real-valued mutation dynamics.  The sign coherence phenomenon has also been extended partially to the setting of exchange matrices with real weights in the work of Akagi and Chen \cite{AC}. Ishibashi and Kano have several results related to asymptotic sign coherence when the mutation sequence consists of repeated traversals of a mutation cycle (which they call a mutation loop), see, e.g.,\;\cite{IshibashiKano2021,IshibashiKano2024}. % In particular, Corollary 5.7  
Reading and Speyer \cite{ReadingSpeyer} developed a general approach to studying sign coherence (and other fundamental properties) in finite and affine type using frameworks, arising from Cambrian combinatorics and sortable elements.

One fruitful direction in the study of quivers with a principal framing are \emph{reddening} sequences (sequences of mutations that produce a quiver with only red vertices) and \emph{maximal green} sequences (sequences of mutations that mutate only at green vertices).
Such sequences have many applications and structural results.
Reddening sequences are mutation cycles on the mutable part \cite{BDP}, and can be used to construct long mutation cycles \cite{MCR}. The existence of a reddening sequence is a mutation invariant \cite{Muller_2016}.
Maximal green sequences give a way to compute Donaldson-Thomas invariants and the twist automorphism of the cluster algebra \cite{KellerGreenSurvey}, and furthermore imply that the full Fock-Goncharov dual basis conjecture holds for the associated cluster variety \cite{GHKK}.
Associated statistics were studied by Ervin in \cite{ErvinRS}.% Because any quiver mutation equivalent to a quiver with all red (resp. green) vertices is sign-coherent, we do not consider these sequences directly. 

\textbf{Structure of the paper.} 
In \Cref{sec: prelim} we establish notation and review relevant results. 
%nobody can skip it; the notation is new.
In \Cref{sec:fork eventual coherence} we build a framework for studying ice forks, and we show in \Cref{thm:Asym Sign Coherence} that ice forks are eventually sign-coherent on many mutation sequences.
We prove our main result (\Cref{thm: geq 3 eventually sc}) in \Cref{sec:wander} and additionally establish related results on the prevalence of ice forks (\Cref{prop: ice forks are all the class} and \Cref{prop: ice forks get wandered into}).
In \Cref{sec:brog} we introduce brog quivers, which generalize several properties of ice forks, and prove that these quickly become eventually sign-coherent on most mutation sequences. In \Cref{sec:low rank} we prove that the asymptotic sign coherence conjecture holds for all mutation-abundant rank $3$ quivers and provide additional examples of quivers which are sign-coherent in similar generality.  Finally, in \Cref{sec:asym vs eventual} we state a version of the conjecture of Gekhtman and Nakanishi (\Cref{conj: asymptotic sign coherence}) and show it is equivalent to a reformulation we call the \emph{eventual sign coherence} conjecture (\Cref{conj: eventual sign coherence}), which more closely aligns with our results. %We conclude by discussing potential further directions.

\section*{Acknowledgments}

We thank Niven Achenjang, Tucker Ervin, Sergey Fomin, Misha Gekhtman, Benjamin Grant, Joel Kamnitzer, Ian Le, Tomoki Nakanishi, Nathan Reading, David Speyer, and Hugh Thomas for helpful conversations. 
Both authors are grateful to LACIM (at Universit\'e du Qu\'ebec \`a Montr\'eal) for providing an excellent and supportive atmosphere to finish this project. We additionally thank Thomas Br\"ustle, Amit Ghosh, and Peter Sarnak for correspondence on a much earlier iteration of this project.

\section{Preliminaries}
\label{sec: prelim}

We begin by recalling some background on quiver mutation, types of mutation sequences, and sign coherence.  We also review some of Warkentin's results on forks  \cite{Warkentin}.

\subsection{Quivers and mutation}

Cluster algebras have become a major topic of study in combinatorics, algebra, representation theory, algebraic geometry, mathematical physics, and more (see \cite{FWZ} for more information). Quivers and quiver mutation play a central role in the construction of cluster algebras.  Cluster algebras with coefficients, in addition to the usual cluster algebra structure, include a coefficient semifield generated by frozen (or coefficient) variables.  Combinatorially, this additional data corresponds to the inclusion of new ``frozen'' vertices in the quiver, which we are forbidden from mutating.  In this paper we look at quiver mutation from a purely combinatorial perspective and do not consider the associated cluster variables.

\begin{defn}
A \emph{quiver} is a finite directed graph with no oriented $1$- or $2$-cycles (though parallel arrows are allowed).
The edges of a quiver are called \emph{arrows}.
The \emph{signed adjacency matrix} $B_Q = (b_{ij})$ of a quiver $Q$ records the number of arrows between each pair of vertices, with $b_{ij} = -b_{ji} >0$ when there is an arrow oriented $i \rightarrow j$ in $Q$. When $b_{ij} \geq 0$, we refer to $b_{ij}$ as the \emph{weight} of the arrow from $i$ to $j$, sometimes represented pictorially by a single arrow with label $b_{ij}$.

The vertices of a quiver are partitioned into sets of \emph{frozen} and \emph{mutable} vertices. (We only mutate at mutable vertices, see Definition~\ref{def:mutation}.) The subquiver of mutable vertices is denoted $Q^{\mut}$. 
The \emph{rank} of a quiver is the number of mutable vertices, which we often denote by $n$.

A \emph{subquiver} of $Q$ is a vertex-induced subgraph, in other words, the quiver obtained by deleting a subset of the vertices of $Q$ and all arrows involving any deleted vertex. We write $Q|_S$ for the subquiver of $Q$ with vertex set $S$.
\end{defn}

Following convention, we depict frozen vertices as squares and mutable vertices as circles.  For the purposes of this paper, we ignore the data of arrows between frozen vertices.

It will often be useful to describe the orientation of the arrows of a quiver, without specifying the weights.  If we draw a directed graph without any labels on the arrows, then the weights may be arbitrary positive integers.
In particular, we will draw $1 \stackrel{1}{\rightarrow} 2$ when we have a single arrow between vertices $1$ and $2$. 
We will use a dashed line to denote the possible presence of an arrow with arbitrary integer weight and orientation. 

\begin{rem}
We will generally treat quivers as labeled graphs, so that two quivers are equal only when they have the same vertex labels. 
\end{rem}

\pagebreak[2]

\begin{defn}
\label{def:mutation}
To \emph{mutate} a quiver $Q$ at a mutable vertex $j$, perform the following steps:
\begin{itemize}
\item for each path $i \stackrel{a}{\rightarrow} j \stackrel{b}{\rightarrow} k$, add $ab$ new arrows $i \stackrel{ab}{\rightarrow} k$ (i.e., add $ab$ to $b_{ik}$); 
\item reverse each arrow incident to $j$;
\item remove oriented $2$-cycles, one-by-one.
\end{itemize}
We call the resulting quiver $\mu_j(Q)$.
Note that each mutation $\mu_j$ is an involution, and that mutation commutes with a global reversal of arrows.
\end{defn}

\begin{definition}
\label{def:visit seq}
A \emph{mutation sequence} $\mathbf M = m_1 m_2 \cdots$ of a quiver $Q$ is a (possibly infinite) sequence of mutable vertices in $Q$.
We inductively define the sequence of quivers: 
\[Q_{\mathbf M}^{(0)} = Q;\]
\[Q_{\mathbf M}^{(i)} = \mu_{m_i}(Q_{\mathbf M}^{(i-1)}).\]

If $\mathbf M$ is finite we have additional notation.   We use $|\mathbf M|$ for the number of mutations in the mutation sequence, and let $Q_{\mathbf M} = Q_{\mathbf M}^{(|\mathbf M|)}$.  We denote by $\overline{\mathbf{M}}$ the infinite periodic sequence which repeats $\mathbf M$ over and over. 
\end{definition}

Two quivers are \emph{mutation-equivalent} if one can be obtained from the other via a finite sequence of mutations, and their \emph{mutation distance} is the minimum length of such a mutation sequence. The \emph{mutation class} of a quiver is the set of all mutation-equivalent quivers.  

\subsection{Sign Coherence}

\begin{definition}\label{defn: red green}
Let $i$ be a mutable vertex adjacent to at least one frozen vertex in a quiver $Q$. 
We say $i$ is \emph{red} (resp. \emph{green}) if all arrows between $i$ and the frozen vertices are oriented towards (resp. away from) $i$.  
If $i$ is red or green, we say that the vertex $i$ is \emph{sign-coherent}.  
A quiver is called \emph{sign-coherent} if all its mutable vertices are sign-coherent. 
\end{definition}

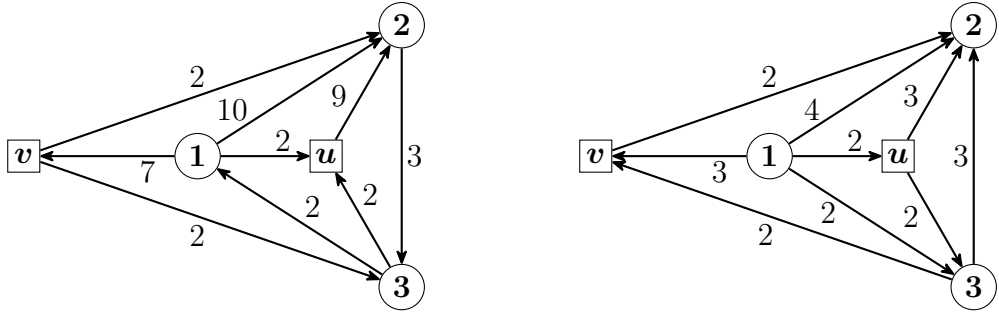
\begin{figure}
\centering
\begin{minipage}{0.45\textwidth}
\begin{tikzpicture}[->, >={Stealth[round]}, node distance=3cm and 4cm, main/.style = {circle, draw, fill=none, minimum size=6mm, inner sep=0pt}, every edge/.append style={thick}, square/.style={regular polygon,regular polygon sides=4}]

\filldraw[black] (0,0)++(180:1.7cm) node[main, fill = white] (1) {\textcolor{black}{$\boldsymbol{1}$}};
\filldraw[black] (0,0)++(60:2cm) node[main, fill = white] (2) {\textcolor{black}{$\boldsymbol{2}$}};
\filldraw[black] (0,0)++(-60:2cm) node[main, fill = white] (3) {\textcolor{black}{$\boldsymbol{3}$}};
\filldraw[black] (0,0) node[main, fill = white, square] (u) {\textcolor{black}{$\boldsymbol{u}$}};
\filldraw[black] (0,0)++(180:4cm) node[main, fill = white, square] (v) {\textcolor{black}{$\boldsymbol{v}$}};

\path
 (1) edge["$10$", outer sep=-3, pos=0.2] (2)
 (2) edge["$3$", outer sep=-2, pos=0.5] (3)
 (3) edge["$2$"', outer sep=-2, pos=0.5] (1)
 (v) edge["$2$"', outer sep=-2, pos=0.5] (3)
 (1) edge["$7$", outer sep=-2, pos=0.2] (v)
 (3) edge["$2$"', outer sep=-2, pos=0.6] (u)
 (u) edge["$9$", outer sep=-2, pos=0.3] (2)
 (1) edge["$2$", outer sep=-2, pos=0.7] (u)
 (v) edge["$2$", outer sep=-2, pos=0.5] (2);
\end{tikzpicture}
\end{minipage}
\begin{minipage}{0.45\textwidth}
\begin{tikzpicture}[->, >={Stealth[round]}, node distance=3cm and 4cm, main/.style = {circle, draw, fill=none, minimum size=6mm, inner sep=0pt}, every edge/.append style={thick}, square/.style={regular polygon,regular polygon sides=4}]

\filldraw[black] (0,0)++(180:1.7cm) node[main, fill = white] (1) {\textcolor{black}{$\boldsymbol{1}$}};
\filldraw[black] (0,0)++(60:2cm) node[main, fill = white] (2) {\textcolor{black}{$\boldsymbol{2}$}};
\filldraw[black] (0,0)++(-60:2cm) node[main, fill = white] (3) {\textcolor{black}{$\boldsymbol{3}$}};
\filldraw[black] (0,0) node[main, fill = white, square] (u) {\textcolor{black}{$\boldsymbol{u}$}};
\filldraw[black] (0,0)++(180:4cm) node[main, fill = white, square] (v) {\textcolor{black}{$\boldsymbol{v}$}};

\path
 (1) edge["$4$", outer sep=-3, pos=0.2] (2)
 (3) edge["$3$", outer sep=-2, pos=0.5] (2)
 (1) edge["$2$"', outer sep=-3, pos=0.3] (3)
 (3) edge["$2$"', outer sep=-13, pos=0.5] (v)
 (1) edge["$3$", outer sep=-2, pos=0.2] (v)
 (u) edge["$2$"', outer sep=-2, pos=0.3] (3)
 (u) edge["$3$", outer sep=-2, pos=0.3] (2)
 (1) edge["$2$", outer sep=-2, pos=0.7] (u)
 (v) edge["$2$", outer sep=-2, pos=0.5] (2);
\end{tikzpicture}
\end{minipage}
\caption{On the left is a rank $3$ quiver with two frozen vertices, $u$ and $v$.  The quiver on the right is obtained by mutation at vertex $3$.}
\label{fig: fork sign coherent}
\end{figure}

\begin{exmp}
Let $Q$ be the quiver on the left in Figure~\ref{fig: fork sign coherent}.
Then vertex $1$ is green, vertex $2$ is red, and vertex $3$ is neither red nor green.
The signed adjacency matrix of the mutable part $Q^{\mut}$ is 
$$B_Q = \begin{pmatrix} 0 & 10 & -2 \\ -10 & 0 & 3 \\ 2 & -3 & 0 \end{pmatrix}.$$
\end{exmp}

\begin{rem}
Discussions of sign coherence sometimes refer to $c$-vectors rather than the frozen vertices in a quiver. 
The $c$-vector corresponding to a mutable vertex $i$ records the (signed) number of arrows from $i$ to $u$ for each frozen vertex $u$.  Equivalently, a quiver is sign-coherent if and only if each $c$-vector is nonzero and has all nonnegative or all nonpositive entries. 
For example, the $c$-vector of vertex $1$ in Figure~\ref{fig: fork sign coherent} is $c_1 = ((c_1)_u,(c_1)_v) = (2,7)$. 
\end{rem}

Red and green vertices are important in a variety of applications of quiver combinatorics \cite{KellerGreenSurvey}.  A fundamental property of cluster algebras is the \emph{sign coherence phenomenon} for quivers \cite{DWZ2}, which states that if all vertices in a quiver are red (or, by global reversal of arrows, if all are green), then every mutation equivalent quiver will be sign-coherent.  Note that this is equivalent to the principal framing formulation of sign coherence (as stated in the introduction) via copying and combining frozen vertices.

\subsection{Forks}
The mutation graph $\Gamma$ of a mutation class is the undirected graph with a vertex for each quiver and an edge $Q \stackrel{i}{\text{---}} Q'$ whenever $Q = \mu_i(Q')$.
Most of the mutation graph of a mutation-infinite quiver is covered by infinite rooted $(n-1)$-ary trees composed of well understood quivers, called forks (Definition~\ref{def:fork} below).

\begin{defn}
A quiver is \emph{connected} if its mutable part is connected as an undirected graph, and there are no isolated frozen vertices.
\end{defn}

\begin{defn}
\label{def:acyclic}
A quiver is \emph{acyclic} if it is acyclic as a directed graph. 
A quiver is \emph{mutation-acyclic} if it is mutation equivalent to an acyclic quiver.
\end{defn}

\begin{defn}
\label{def:elbow-oriented}
In the special case of $3$-vertex quivers, we have additional terminology. 

Suppose that the underlying directed graph of a $3$-vertex quiver $Q$ is acyclic. Then there is at most one vertex $j$ such that $j$ is in the middle of an oriented path $i \rightarrow j \rightarrow k$. If at least one of $i$ or $k$ is mutable, then we call $j$ an \emph{elbow} in $Q$.

A $3$-vertex quiver is an \emph{oriented cycle} if its underlying directed graph has at most one frozen vertex and is not acyclic; we will only use this terminology for $3$-vertex quivers.
\end{defn}

\begin{figure}
\centering
\begin{tikzpicture}[->, >={Stealth[round]}, node distance=3cm and 4cm, main/.style = {circle, draw, fill=none, minimum size=6mm, inner sep=0pt}, every edge/.append style={thick}, square/.style={regular polygon,regular polygon sides=4}]

\filldraw[black] (0,0)++(45:1cm) node[main, fill = white] (1) {\textcolor{black}{$\boldsymbol{1}$}};
\filldraw[black] (0,0)++(-45:1cm) node[main, fill = white] (2) {\textcolor{black}{$\boldsymbol{2}$}};
\filldraw[black] (0,0)++(-135:1cm) node[main, fill = white] (3) {\textcolor{black}{$\boldsymbol{3}$}};

\path
 (1) edge["4"] (2)
 (2) edge["1"] (3);

\begin{scope}[shift = {(5,0)}]
\filldraw[black] (0,0) node[main, fill = white] (a) {\textcolor{black}{$\boldsymbol{4}$}};
\filldraw[black] (0,0)++(90:3cm) node[main, fill = white] (1) {\textcolor{black}{$\boldsymbol{1}$}};
\filldraw[black] (0,0)++(-30:3cm) node[main, fill = white] (2) {\textcolor{black}{$\boldsymbol{2}$}};
\filldraw[black] (0,0)++(210:3cm) node[main, fill = white] (3) {\textcolor{black}{$\boldsymbol{3}$}};

\path
 (1) edge["2"'] (3)
 (2) edge["9"'] (1)
 (3) edge["4"'] (2)
 (a) edge["12", pos=0.3] (1)
 (2) edge["13"', pos=0.7] (a)
 (a) edge["7", pos=0.3] (3);
\end{scope}

\begin{scope}[shift = {(10,0)}]
\filldraw[black] (0,0)++(90:1cm) node[main, fill = white] (1) {\textcolor{black}{$\boldsymbol{1}$}};
\filldraw[black] (0,0)++(-30:1cm) node[main, fill = white] (2) {\textcolor{black}{$\boldsymbol{2}$}};
\filldraw[black] (0,0)++(210:1cm) node[main, fill = white] (3) {\textcolor{black}{$\boldsymbol{3}$}};

\path
 (1) edge["2"'] (3)
 (2) edge["2"'] (1)
 (3) edge["2"'] (2);
\end{scope}
\end{tikzpicture}

\caption{Three connected quivers with no frozen vertices. From left to right, an acyclic quiver which has elbow $2$ and is not complete; an abundant quiver which is a fork with point of return $3$; and the Markov quiver (which is an oriented cycle).}
\label{fig:quiver examples}
\end{figure}
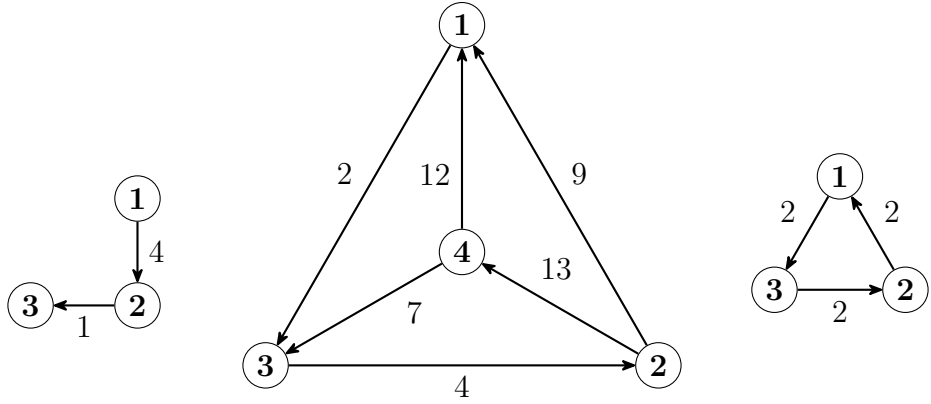

\begin{defn} 
\label{def:complete-abundant}
A quiver is \emph{complete} (resp. \emph{abundant}) if it has at least $1$ arrow (resp. at least $2$ arrows) between every pair of vertices where at least one of the two is mutable.
\end{defn}

Given a quiver $Q$ with a vertex $v$, let $Q^-(v)$ be the set of vertices with arrows pointing toward $v$ and let $Q^+(v)$ be the set of vertices with arrows coming from $v$.

\begin{defn}[{\cite[Definition 2.1]{Warkentin}}]
\label{def:fork}
A \emph{fork} is an abundant, non-acyclic quiver $F$ with at most one frozen vertex and a distinguished vertex $r$ (called the \emph{point of return}) such that 
\begin{itemize}
\item For all $i \in F^+(r)$ and $j \in F^-(r)$, we have $b_{ij} > b_{ri}$ and $b_{ij} > b_{jr}$.
\item The full subquivers induced by $F^+(r)$ and $F^-(r)$ are acyclic.
\end{itemize}
\end{defn}

Note that the point of return is uniquely determined.  

\begin{exmp}
\label{eg: it is an ice fork}
Let $Q$ be the quiver on the left in Figure~\ref{fig: fork sign coherent}.
Then $Q$ is connected, complete, and abundant. The mutable part $Q^{\mut}$ is an oriented cycle.
Although $Q$ is not a fork (as it has two frozen vertices), both of the subquivers $Q|_{123u}$ and $Q|_{123v}$ are forks with point of return $3$.
The quiver $\mu_3(Q)^{\mut}$ is acyclic and vertex $3$ is its elbow.
\end{exmp}

\begin{rem}
Forks were initially defined for quivers with only mutable vertices.
When a quiver has two frozen vertices, it is not always clear how to interpret the inequalities in the definition.
\end{rem}

%\begin{exmp}
%Suppose that $Q$ is a rank 3 abundant acyclic quiver with elbow $j$ (Definitions~\ref{def:elbow-oriented} and~\ref{def:complete-abundant}).
%Then $\mu_j(Q)$ is a fork with point of return $j$.
%\end{exmp}

Crucially, forks are ``forks'' in the mutation graph in the following sense.

\begin{lem}[Tree Lemma, {\cite[Lemma 2.8]{Warkentin}}]
\label{lem:tree lemma}
Let $\Gamma$ be the mutation graph of a rank $n$ quiver and assume that one vertex of $\Gamma$ is a fork $Q$ with point of return $r$.  When we delete the edge $e$ between $Q$ and $\mu_r(Q)$, the connected component $\Gamma'$ of $\Gamma - e$ containing $Q$
 is an infinite complete $(n-1)$-ary tree rooted at $Q$. \end{lem}

\subsection{Mutation Sequences}
There are many natural restrictions one can put on a mutation sequence, both in terms of local and global behavior.  We will explore when eventual sign coherence holds in the presence or absence of these restrictions. 

\begin{defn}
A (potentially infinite) mutation sequence $\mathbf M$ is called 
\begin{itemize}
    \item \emph{reduced} if no vertex is mutated at twice in a row, $m_i \neq m_{i+1}$;
    \item \emph{simple} if every quiver it visits is distinct, $Q^{(i)}_{\mathbf M} \neq Q^{(j)}_{\mathbf M}$ for $i \neq j$; 
    \item \emph{monotone} if every mutation in the sequence increases the minimum number of mutations required to return $Q^{(i)}_{\mathbf M}$ to the initial quiver $Q$.
\end{itemize}
\end{defn}

Note that a monotone mutation sequence is necessarily simple, and a simple mutation sequence is necessarily reduced.

\begin{rem}
A mutation sequence $\mathbf M$ at $Q$ identifies a walk in the mutation graph:
$$ Q \stackrel{m_1}{\text{---}} Q^{(1)}_{\mathbf M} \stackrel{m_2}{\text{---}} \cdots$$
If $\mathbf M$ is simple, then this walk is a path. 
If $\mathbf M$ is reduced, then we never traverse the same edge twice in a row. 
If $\mathbf M$ is monotone, then the walk is a `geodesic' ray.
\end{rem}

\begin{rem}
Given a quiver $Q' \neq Q$ inside of $\Gamma - e$, the shortest (and only reduced) mutation sequence that reaches $Q$ begins by mutating at the point of return of $Q'$. 
Thus it is easy to mutate out of these $(n-1)$-ary trees of forks: 
repeatedly find and mutate at the point of return.
\end{rem}

\begin{exmp}
\label{eg: dead end}
Not all finite monotone (resp. simple) paths can be extended to an infinite monotone (resp. simple) path. Consider the (mutation infinite) quiver $Q$ shown in \Cref{fig: not all simple}.
Every mutation of $\mu_2(Q)$ produces a quiver which is at most one mutation from $Q$. Thus the monotone sequence $2$ cannot be extended to an infinite monotone sequence.
Similarly, if we start at $Q'=\mu_1(Q)$, then the simple mutation sequence $\mathbf{M} = 131$ cannot be extended to an infinite simple mutation sequence.
%These issues persist if one treats quivers up to isomorphism (instead of distinguishing quivers by their labels).
\end{exmp}

For infinite mutation sequences, we can also consider how often each vertex is mutated at.

\begin{defn}
An infinite mutation sequence $m_1m_2\cdots$ is called 
\begin{itemize}
    \item \emph{weakly balanced} if every mutable vertex is mutated at infinitely many times.
    \item \emph{balanced} if for every mutable vertex $k$, we have
    $$\liminf_{\ell \to \infty} \frac{1}{\ell}\left|\{j \in [1,\ell] : m_j = k\}\right| > 0\,.$$
\end{itemize}
\end{defn}

\begin{exmp}
    Let $Q$ be a rank $3$ quiver with mutable vertices $1,2,3$.  The mutation sequence $\overline{12}$ is not weakly balanced, the mutation sequence $123\overline{12}$ is weakly balanced but not balanced, and the mutation sequence $\overline{1213}$ is balanced.
\end{exmp}

\begin{figure}
\centering
\begin{tikzpicture}[->, >={Stealth[round]}, node distance=3cm and 4cm, main/.style = {circle, draw, fill=none, minimum size=6mm, inner sep=0pt}, every edge/.append style={thick}, square/.style={regular polygon,regular polygon sides=4}]

\filldraw[black] (-2,0) node[main, fill = white] (1) {\textcolor{black}{$\boldsymbol{1}$}};
\filldraw[black] (0,0) node[main, fill = white] (2) {\textcolor{black}{$\boldsymbol{2}$}};
\filldraw[black] (2,0) node[main, fill = white] (3) {\textcolor{black}{$\boldsymbol{3}$}};
\path
 (1) edge["$2$"] (2)
 (3) edge["$4$"'] (2);

\begin{scope}[shift = {(6,0)}]
\filldraw[black] (0,0) circle (4pt) coordinate (q) node[above left=-1pt] {Q};
\filldraw[black] (1,1) circle (4pt) coordinate (1);
\filldraw[black] (1,-1) circle (4pt) coordinate (2);
\filldraw[black] (2,0) circle (4pt) coordinate (3);
\coordinate (5) at (-1,0);
\coordinate (6) at (3,0);

\path
 (1) edge[-, "3"] (3)
 (2) edge[-, "2"] (1)
 (3) edge[-, "1"] (2)
 (q) edge[-, "1"] (1)
 (q) edge[-, "3"'] (2)
 (q) edge[-, "2"] (5)
 (3) edge[-, "2"] (6);
\end{scope}
\end{tikzpicture}
\caption{On the left, a $3$-vertex acyclic quiver $Q$. On the right, a segment of its mutation graph.}
\label{fig: not all simple}
\end{figure}
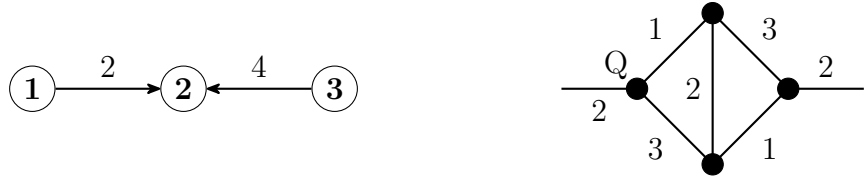

We are primarily interested in sign coherence phenomena on infinite mutation sequences.  We say that the mutation sequence $\mathbf M$ on $Q$ is \emph{eventually sign-coherent} if every quiver $Q_{\mathbf M}^{(j)}$ is sign-coherent for sufficiently large $j$.

%Forks can be generalized, without changing the properties too much. 

%\begin{defn}[{\cite[Definition 3.7]{Warkentin}}]
%A \emph{pre-fork} is a quiver $P$ with two distinct vertices $k,k'$ such that
%\begin{itemize}
%    \item Both full subquivers $P \setminus \{k\}$ and $P \setminus \{k'\}$ are forks with a common point of return.
%    \item Any vertex $i \neq k,k'$ is either in $P^+(k) \cup P^+(k')$ or in $P^-(k) \cup P^-(k')$.
%\end{itemize}
%\end{defn}

%An analog of the Tree Lemma holds for pre-forks.

\subsection{Vortices}
Vortices are a type of (sub)quiver that can behave poorly with respect to certain mutation properties.  We require this notion for our work on brog quivers (\Cref{sec:brog}), where we avoid mutations at any vertex that is the apex of a vortex subquiver.  

\begin{defn}[{\cite[Definition~6.2]{LMC}}]
\label{def:vortex}
A \emph{vortex} is a quiver on four vertices, at least three of which are mutable, such that one vertex is either a source or a sink and the other three vertices support an oriented cycle. 
We call the unique source/sink the \emph{apex}.
A quiver is \emph{vortex-free} if no $4$-vertex subquiver is a vortex. 
\end{defn}

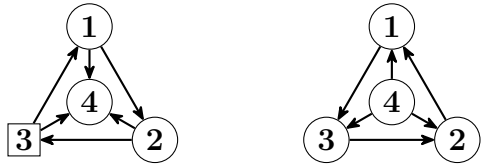
\begin{figure}[ht]
\centering
\begin{tikzpicture}[->, >={Stealth[round]}, node distance=3cm and 4cm, main/.style = {circle, draw, fill=none, minimum size=6mm, inner sep=0pt}, every edge/.append style={thick}, square/.style={regular polygon,regular polygon sides=4}]

\filldraw[black] (0,0) node[main, fill = white] (a) {\textcolor{black}{$\boldsymbol{4}$}};
\filldraw[black] (0,0)++(90:1cm) node[main, fill = white] (1) {\textcolor{black}{$\boldsymbol{1}$}};
\filldraw[black] (0,0)++(-30:1cm) node[main, fill = white] (2) {\textcolor{black}{$\boldsymbol{2}$}};
\filldraw[black] (0,0)++(210:1cm) node[main, fill = white, square] (3) {\textcolor{black}{$\boldsymbol{3}$}};

\path
 (1) edge (2)
 (2) edge (3)
 (3) edge (1)
 (1) edge (a)
 (2) edge (a)
 (3) edge (a);

\begin{scope}[shift = {(4,0)}]
\filldraw[black] (0,0) node[main, fill = white] (a) {\textcolor{black}{$\boldsymbol{4}$}};
\filldraw[black] (0,0)++(90:1cm) node[main, fill = white] (1) {\textcolor{black}{$\boldsymbol{1}$}};
\filldraw[black] (0,0)++(-30:1cm) node[main, fill = white] (2) {\textcolor{black}{$\boldsymbol{2}$}};
\filldraw[black] (0,0)++(210:1cm) node[main, fill = white] (3) {\textcolor{black}{$\boldsymbol{3}$}};

\path
 (1) edge (3)
 (2) edge (1)
 (3) edge (2)
 (a) edge (1)
 (a) edge (2)
 (a) edge (3);
\end{scope}
\end{tikzpicture}
\caption{Possible orientations of two vortices with apex $4$.} 
\label{fig: vortex eg}
\end{figure}

The following can be derived via short computations from the definitions.

\begin{prop}
\label{prop:apex preserves}
Suppose $j$ is the apex of a vortex in $Q$.
Then $j$ is still the apex of a vortex in $\mu_j(Q)$.
\end{prop}

\begin{prop}
\label{lem:forks are vortex free}
If $Q$ is a fork then $Q$ is vortex-free.
\end{prop}

Since forks make up most of the mutation graph, the above result ensures that avoiding mutation at an apex of a vortex is often not restrictive.

\section{Forks and eventual sign coherence}
\label{sec:fork eventual coherence}

We extend the Warkentin's notion of forks (\Cref{def:fork}) to allow for more frozen vertices, yielding quivers we call ``ice forks''.  We show that ice forks enjoy many of the same properties as forks.  Moreover, they are eventually sign-coherent under almost all mutation sequences, which is crucial to our main result (\Cref{thm: geq 3 eventually sc}).

\begin{definition}
\label{def:ice fork}
Let $Q$ be a quiver with mutable vertices $[n]$ and frozen vertices $u_1, u_2, \ldots u_m$ (with $m\geq 1$). 
We say $Q$ is an \emph{ice fork} with point of return $r$ if 
each of the subquivers $Q|_{[n] \cup u_i}$ are forks with common point of return $r$. 
\end{definition}

\begin{exmp}
As we nearly observed in \Cref{eg: it is an ice fork}, the quiver on the left in \Cref{fig: fork sign coherent} is an ice fork with point of return~$3$. 
\end{exmp}

The purpose of this section is to prove \Cref{thm:Asym Sign Coherence}, which resolves a special case of Conjecture~\ref{conj: eventual sign coherence} and will be useful in establishing our main results.

\begin{theorem}
\label{thm:Asym Sign Coherence}
Let $Q$ be a rank $n$ ice fork and $\mathbf M = m_1 m_2 \cdots $ a reduced mutation sequence such that $m_1$ is not the point  of return and for some $i \in \N$, we have $[n] \subseteq \{m_k : k \leq i\}$.
Then for $j > i$, the quiver $Q^{(j)}_{\mathbf M}$ is sign-coherent. 
\end{theorem}

Lemmas~\ref{lem: skip colored}, \ref{lem: colors stay}, and \ref{lem:mutation makes more color} describe how ascents increase the number of sign-coherent vertices, and control how they change the weights between (1) mutable vertices which are not sign-coherent and (2) frozen vertices.

\begin{defn}
\label{def:cycle-preserving}
A mutation $j$ is \emph{cycle-preserving} for a quiver $Q$ if whenever $Q|_{ijk}$ is an oriented 3-cycle containing $j$ then so is $\mu_j(Q)|_{ijk}$. %every oriented 3-cycle through $j$ in $Q$ is still an oriented 3-cycle in $\mu_j(Q)$.
%note there is no such thing as 4-cycle or 5-cycle preserving.
A mutation sequence $\mathbf M = m_1 m_2 \cdots $ is cycle-preserving if $m_{\ell}$ is cycle-preserving for $Q_{\mathbf M}^{(\ell-1)}$ for all $\ell \geq 1$. 
\end{defn}

\begin{rem}
When $j$ is cycle-preserving the orientations of $\mu_j(Q)$ are entirely determined by the orientations of $Q$.
This property captures a broader principle, that much of quiver mutation combinatorics is captured by the orientations of the arrows in the quiver, with all `large' weights $b_{ij}$ behaving essentially the same.
See also the discussion in the introduction of \cite{Warkentin}, which attributes a related conjecture to Dieter Happel \cite{happel2009reconstruction}. %\cS{Happel's question is if the mutation class of an acyclic quiver is determined by 'basic equality', that is treat all weights $>2$ as $2$, does this operation change the mutation (or exchange) graph? That fork mutations are cycle preserving is a refinement/consequence/manifestation of 'yes'.}
\end{rem}

\begin{rem}
In particular, if $j$ is cycle-preserving, then $j$ cannot be the elbow of a $3$-vertex acyclic subquiver of $\mu_j(Q)$.
\end{rem}

\begin{lem}
\label{lem:forks preserve}
Suppose that $\mu_i(Q)$ is a fork (resp. ice fork) with point of return $i$.
Then $i$ is cycle-preserving for $Q$.
\end{lem}

\begin{proof}
Suppose $Q|_{ijk}$ is an oriented cycle with $b_{ij} > 0$.
Mutation at $i$ will reverse the arrows between $i$ and $j,k$. 
As $\mu_i(Q)$ is a fork with point of return $i$, we must have $k \rightarrow j$ in $\mu_i(Q)$. So $\mu_j(Q)|_{ijk}$ is again an oriented cycle.
\end{proof}

By Lemma~\ref{lem:tree lemma}, we have the following: 

\begin{cor}
\label{cor: cycle-preserving tree lem}
If $Q$ is a quiver with a vertex $i$ such that $\mu_i(Q)$ is a fork with point of return~$i$, then every reduced mutation sequence $\mathbf M = i m_2 \cdots$ is cycle-preserving for $Q$.
\end{cor}

Incidentally, this lets us construct examples of quivers with only cycle-preserving mutation sequences.

\begin{cor}
\label{cor: if all forks then all sequences preserve all cycles}
Suppose $Q$ is a quiver such that for every mutable vertex $i$, the quiver $\mu_i(Q)$ is an ice fork with point of return $i$. 
Then every reduced mutation sequence is cycle-preserving for $Q$.
\end{cor}

\begin{figure}
\centering
\begin{tikzpicture}[->, >={Stealth[round]}, node distance=3cm and 4cm, main/.style = {circle, draw, fill=none, minimum size=6mm, inner sep=0pt}, every edge/.append style={thick}, square/.style={regular polygon,regular polygon sides=4}]

\begin{scope}[shift = {(-4,0)}]
\filldraw[black] (-0.8,0) node[main, fill = white] (1) {\textcolor{black}{$\boldsymbol{1}$}};
\filldraw[black] (0.8,0) node[main, fill = white] (2) {\textcolor{black}{$\boldsymbol{2}$}};

\path
 (1) edge["$3$"] (2);
\end{scope}

\filldraw[black] (0,0)++(90:1cm) node[main, fill = white] (1) {\textcolor{black}{$\boldsymbol{1}$}};
\filldraw[black] (0,0)++(-30:1cm) node[main, fill = white] (2) {\textcolor{black}{$\boldsymbol{2}$}};
\filldraw[black] (0,0)++(210:1cm) node[main, fill = white, square] (3) {\textcolor{black}{$\boldsymbol{u}$}};

\path
 (1) edge["$10$"] (2)
 (2) edge["$3$"] (3)
 (3) edge["$2$"] (1);

\begin{scope}[shift = {(4,0)}]
\filldraw[black] (0,0)++(90:1cm) node[main, fill = white] (1) {\textcolor{black}{$\boldsymbol{1}$}};
\filldraw[black] (0,0)++(-30:1cm) node[main, fill = white] (2) {\textcolor{black}{$\boldsymbol{2}$}};
\filldraw[black] (0,0)++(210:1cm) node[main, fill = white] (3) {\textcolor{black}{$\boldsymbol{3}$}};

\path
 (1) edge["$7$"] (2)
 (2) edge["$6$"] (3)
 (3) edge["$5$"] (1);
\end{scope}

\begin{scope}[shift = {(9,0)}]
\filldraw[black] (0,0) node[main, fill = white] (a) {\textcolor{black}{$\boldsymbol{4}$}};
\filldraw[black] (0,0)++(90:2cm) node[main, fill = white] (1) {\textcolor{black}{$\boldsymbol{1}$}};
\filldraw[black] (0,0)++(-30:2cm) node[main, fill = white] (2) {\textcolor{black}{$\boldsymbol{2}$}};
\filldraw[black] (0,0)++(210:2cm) node[main, fill = white] (3) {\textcolor{black}{$\boldsymbol{3}$}};

\path
 (1) edge["$3$"'] (3)
 (2) edge["$3$"'] (1)
 (3) edge["$3$"'] (2)
 (a) edge["$3$", pos=0.3] (1)
 (a) edge["$7$", pos=0.2, outer sep=-2] (2)
 (3) edge["$3$", pos=0.8, outer sep=-2] (a);
\end{scope}
\end{tikzpicture}
\caption{Four quivers with only cycle-preserving mutation sequences.}
\label{fig: cycle preserving egs}
\end{figure}
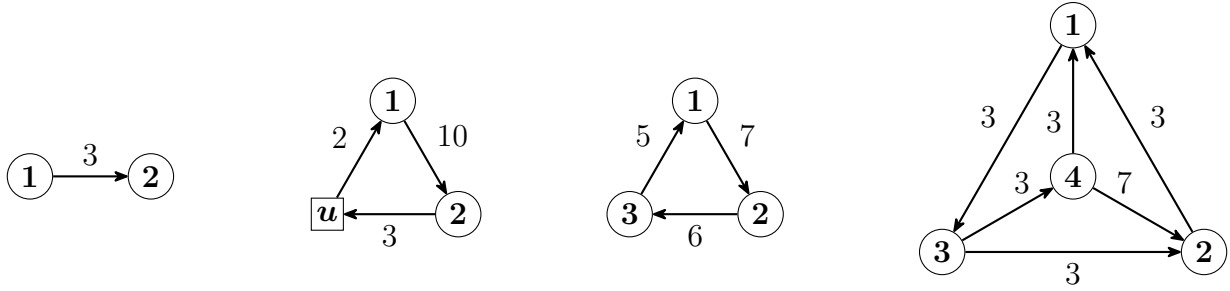

\begin{exmp}
By Corollary~\ref{cor: if all forks then all sequences preserve all cycles}, any reduced mutation sequence on the rightmost three quivers in Figure~\ref{fig: cycle preserving egs} is cycle-preserving.
This is vacuously true for the leftmost quiver. 
\end{exmp}

\begin{prop}
\label{prop: cycle preserving is like a por}
Let $Q$ be a complete quiver. 
Suppose that $j$ is cycle-preserving and not the apex of a vortex subquiver of $Q$. Then every oriented $3$-cycle in $\mu_j(Q)$ includes $j$ and $\mu_j(Q)$ is vortex-free.
\end{prop}

\begin{proof}
Consider a $4$-vertex subquiver $Q|_{ijk\ell}$.  Since $j$ is not the apex of a vortex, there must be two directed $2$-paths through $j$ in $Q|_{ijk\ell}$.  Without loss of generality (possibly using global reversal of arrows), say these $2$-paths are $i \rightarrow j \rightarrow k$ and $\ell \rightarrow j \rightarrow k$.  Then we have $i \rightarrow k$ and $\ell \rightarrow k$ in $\mu_j(Q)$, so $\mu_j(Q)|_{ik\ell}$ is not an oriented $3$-cycle.

We have established that any cycle in $\mu_j(Q)$ includes $j$.  Fix any $3$ cycle $i \rightarrow j \rightarrow k \rightarrow i$ in $\mu_j(Q)$, and we will show that each vertex $\ell$ cannot be the apex of a vortex on $i,j,k,\ell$.  If $\ell \rightarrow j$ in $Q$, then $\ell \rightarrow i$ in $\mu_j(Q)$ but $\ell \leftarrow j$, so $\ell$ is not the apex of a vortex.  On the other hand, if $\ell \leftarrow j$ in $Q$, then we have $\ell \rightarrow j$ and $\ell \leftarrow k$ in $\mu_j(Q)$, so $\ell$ is again not the apex of a vortex.  
\end{proof}

\begin{rem}
\Cref{cor: cycle-preserving tree lem} and \Cref{prop: cycle preserving is like a por} show that cycle-preserving mutations generalize mutations of forks that do not occur at the point of return.
However, cycle-preserving mutations are more general, including (for example) sink or source mutations in an acyclic quiver. 
\end{rem}

We now turn to proving our key lemmas for \Cref{thm:Asym Sign Coherence}.
By \Cref{cor: cycle-preserving tree lem}, it suffices to describe how cycle-preserving mutations change the color (or lack there-of) of the vertices. 

\begin{lem}
\label{lem: skip colored}
%Let $R$ be an ice fork. 
%Suppose $j$ is an ascent of $R$ which is red or green. 
%Then every red or green vertex in $R$ is also red or green in $\mu_j(R)$. 
%For each mutable vertex $k$ which is neither red nor green, either
%the weights between the frozens and $k$ are unchanged, or
%$k$ becomes red or green.
Let $Q$ be a complete quiver and $j$ a cycle-preserving vertex which is red or green. % and adjacent to all other vertices.
Then every red or green vertex of $Q$ is red or green in $\mu_j(Q)$. 
For each mutable vertex $k$ which is neither red nor green, either the weights between $k$ and the frozen vertices are unchanged, or $k$ becomes red or green.
\end{lem}

\begin{proof}
Without loss of generality, say that $j$ is green.
Mutation at $j$ can only add arrows directed towards frozen vertices from mutable vertices $k$ such that $j \leftarrow k$.
Thus the color of~$k$ will be unchanged if $j \rightarrow k$ or if $k$ is already green.
So we may assume that $k$ is not green and $j \leftarrow k$.
%First we show that all green and red vertices are green or red in $\mu_j(Q)$.
%So $j$ preserves all green vertices.
%Suppose instead we have $j \leftarrow k$ for $k$ red.

Fix any such mutable vertex $k$ and a frozen vertex $u$, so we have a path $k \rightarrow j \rightarrow u$ in~$Q$. 
If $k \rightarrow u$ then $j$ is the elbow in $Q|_{jku}$ and $\mu_j(Q)|_{jku}$ is an oriented cycle.
If $k \leftarrow u$ then $Q|_{jku}$ is an oriented cycle, because $j$ is cycle-preserving we conclude that $\mu_j(Q)|_{jku}$ is an oriented cycle.
Thus $k$ is green in $\mu_j(Q)$.
%If $j \rightarrow k$, then $j$ is a source in $R|_{jku}$, so the arrows between $k$ and each frozen vertex are unchanged.
%If $j \leftarrow k$ then $j$ is an ascent in $R|_{jku}$. 
%Therefore $R|_{jku}$ is cyclically oriented in $\mu_j(R)$ with $k \rightarrow u$, and so $k$ is green in $\mu_j(R)$. 
\end{proof}

\Cref{lem: skip colored} gives a short combinatorial proof that sign coherence is preserved by mutations of forks.

\begin{prop}
Suppose $Q$ is a sign-coherent ice fork with point of return $r$.
Then for every mutation sequence $\mathbf M = m_1 m_2 \cdots$ of $Q$ with $m_1 \neq r$, the quiver $Q^{(i)}_{\mathbf M}$ is sign-coherent.
\end{prop}

Thus we can restrict our attention to the cycle-preserving mutations which are neither red nor green.

\begin{lem}
\label{lem: colors stay}
Let $Q$ be an ice fork with $j$ a mutable vertex that is not the point of return. 
Let $k \neq j$ be another mutable vertex in $Q$.
If $k$ is red or green in $Q$ then it is also red or green in $\mu_j(Q)$.
\end{lem}

\begin{proof}
If $j$ is sign-coherent then the claim follows from \Cref{lem: skip colored}. So we may assume $j$ is not red nor green; let $u,v$ be two frozen vertices with $u \rightarrow j \rightarrow v$.
Without loss of generality, assume $j\rightarrow k$ in $Q$.

If $k$ is red we have the following orientations in $Q$:

%\begin{figure}
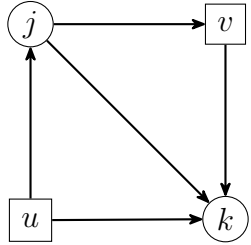
\begin{center}
    
\begin{tikzpicture}[->, >={Stealth[round]}, node distance=3cm and 4cm, main/.style = {circle, draw, fill=none, minimum size=6mm, inner sep=0pt}, every edge/.append style={thick}, square/.style={regular polygon,regular polygon sides=4}]

\node[main] (j) {$j$};
\node[draw, square, inner sep=.8mm] (v) [right=2cm of j] {$v$};
\node[main] (k) [below=2cm of v] {$k$};
\node[draw, square, inner sep=.8mm] (u) [below=2cm of j] {$u$};

\path 
    (u) edge (j)
    (j) edge (v)
    (j) edge (k)
    (u) edge (k)
    (v) edge (k);
\end{tikzpicture}
    %\label{fig:orientations with j uncolored}
%\end{figure}
\end{center}
We can immediately see that $k$ is still red in $\mu_j(Q)$.

We claim that $k$ cannot be green (recall that we are assuming $j\rightarrow k$).  
If $k \rightarrow u$, then $Q|_{jku}$ is cyclically oriented. 
Note that $u$ is frozen and $j$ is assumed not to be the point of return, so $k$ is the point of return in $Q$.
But this contradicts the assumption that $Q$ is an ice fork, as we have $j \rightarrow k \rightarrow v \leftarrow j$ (in particular, $b_{vj} < 0 < b_{kv}$).
\end{proof}

In other words, \Cref{lem: colors stay} tells us that the set of red or green vertices can only grow when we are mutating ice forks.

\begin{lem}
\label{lem:mutation makes more color}
Let $Q$ be a complete quiver with at least two frozen vertices.
Fix two distinct mutable vertices $j,k$.
Suppose $j$ is cycle-preserving for $Q$, and $k$ is cycle-preserving for $\mu_j(Q)$.
If $j$ is neither red nor green, then $j$ is red or green in $\mu_k(\mu_j(Q))$.
\end{lem}

\begin{proof}
Without loss of generality, assume $j \rightarrow k$ in $Q$.
Let $u,v$ be two frozen vertices so that $u \rightarrow j \rightarrow v$ in $Q$. 
We know the following orientations of $\mu_j(Q)|_{jkuv}$: 

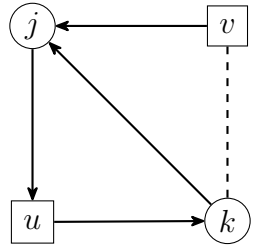
\centering
\begin{tikzpicture}[->, >={Stealth[round]}, node distance=3cm and 4cm, main/.style = {circle, draw, fill=none, minimum size=6mm, inner sep=0pt}, every edge/.append style={thick}, square/.style={regular polygon,regular polygon sides=4}]

\node[main] (j) {$j$};
\node[draw, square, inner sep=.8mm] (v) [right=2cm of j] {$v$};
\node[main] (k) [below=2cm of v] {$k$};
\node[draw, square, inner sep=.8mm] (u) [below=2cm of j] {$u$};

\path 
    (j) edge (u)
    (v) edge (j)
    (k) edge (j)
    (u) edge (k)
    (k) edge[dashed, -] (v);
\end{tikzpicture}

Regardless of the unknown orientation, vertex $j$ is red in $\mu_k(\mu_j(R))$.
\end{proof}

\begin{proof}[{(Proof of \Cref{thm:Asym Sign Coherence})}]
%By \cite[Lemma 6.8]{LMC}, because $\mathbf M$ is an ascent sequence we know $Q^{(1)}$ is an ice fork. 
By \Cref{cor: cycle-preserving tree lem}, each mutation in $\mathbf M$ is cycle-preserving.
By the definition of a fork (and the choice of $\mathbf M$), $Q_{\mathbf M}^{(i)}$ is complete for all $i$.
By Lemmas~\ref{lem: skip colored} and~\ref{lem: colors stay}, if a vertex is red or green in $Q_{\mathbf M}^{(i)}$ then it is red or green in all $Q_{\mathbf M}^{(j)}$ for $j>i$. 
By \Cref{lem:mutation makes more color}, vertex $m_i$ will be red or green in $Q_{\mathbf M}^{(i+1)}$. 
By assumption on $\mathbf M$, there is an $\ell$ such that $[n]\subseteq m_1\cdots m_{\ell}$.
It follows that $Q_{\mathbf M}^{(i)}$ is sign-coherent for $i \geq \ell+1$.
\end{proof}

\begin{rem}
We note that the proof of Theorem~\ref{thm:Asym Sign Coherence} gives a precise bound on the number of mutations required to become sign-coherent in terms of the number of vertices that are initially sign-coherent.
\end{rem}

\begin{exmp}
Let $Q$ be the fork shown in Figure~\ref{fig: fork sign coherent}, and consider the mutation sequence $\mathbf M = \overline{123}$.
In $Q^{(0)}_{\mathbf M}$, vertex $3$ is not red or green, while vertices $1$ and $2$ are.
In $Q^{(1)}_{\mathbf M} = \mu_1(Q)$, vertex $3$ is red, and so all of the vertices are red or green. All vertices will remain red or green from this point forward.
\end{exmp}

\begin{exmp}
Let $Q$ be the fork shown in Figure~\ref{fig: another sign coherent fork}, and consider the mutation sequence $\mathbf M = 3\overline{12}$.
In $Q^{(0)}_{\mathbf M}$, no vertex is red or green.
In $Q^{(1)}_{\mathbf M} = \mu_3(Q)$, vertex $1$ is red, while vertex $3$ and $2$ are not.
In $Q^{(2)}_{\mathbf M} = \mu_1(Q^{(1)}_{\mathbf M})$, vertices $1$ and $3$ are green or red, while vertex $2$ remains neither.
While the situation is unchanged in $Q^{(3)}_{\mathbf M} = \mu_2(Q^{(2)}_{\mathbf M})$, vertex $2$ is green in $Q^{(4)}_{\mathbf M}$ and $Q^{(j)}_{\mathbf M}$ is sign-coherent for $j\geq 4$.
\end{exmp}

\begin{figure}
\centering
\begin{tikzpicture}[->, >={Stealth[round]}, node distance=3cm and 4cm, main/.style = {circle, draw, fill=none, minimum size=6mm, inner sep=0pt}, every edge/.append style={thick}, square/.style={regular polygon,regular polygon sides=4}]

\filldraw[black] (0,0)++(180:1.7cm) node[main, fill = white] (1) {\textcolor{black}{$\boldsymbol{1}$}};
\filldraw[black] (0,0)++(60:2cm) node[main, fill = white] (2) {\textcolor{black}{$\boldsymbol{2}$}};
\filldraw[black] (0,0)++(-60:2cm) node[main, fill = white] (3) {\textcolor{black}{$\boldsymbol{3}$}};
\filldraw[black] (0,0) node[main, fill = white, square] (u) {\textcolor{black}{$\boldsymbol{u}$}};
\filldraw[black] (0,0)++(180:4cm) node[main, fill = white, square] (v) {\textcolor{black}{$\boldsymbol{v}$}};

\path
 (1) edge["$5$", outer sep=-3, pos=0.2] (2)
 (2) edge["$13$", outer sep=-2, pos=0.5] (3)
 (3) edge["$3$"', outer sep=-2, pos=0.5] (1)
 (v) edge["$20$"', outer sep=-2, pos=0.5] (3)
 (1) edge["$7$", outer sep=-2, pos=0.2] (v)
 (v) edge["$2$", outer sep=-2, pos=0.5] (2)
 (3) edge["$2$"', outer sep=-2, pos=0.6] (u)
 (2) edge["$9$"', outer sep=-2, pos=0.7] (u)
 (u) edge["$4$"', outer sep=-2, pos=0.3] (1);
\end{tikzpicture}
\caption{An ice fork with point of return $1$.}
\label{fig: another sign coherent fork}
\end{figure}
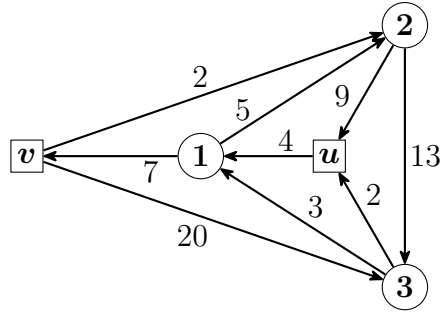

\section{Almost always eventual sign coherence}
\label{sec:wander}

We now show that Gekhtman and Nakanishi's \emph{asymptotic sign coherence conjecture} holds almost always for quivers with frozen vertices.  That is, we show that for any mutation-infinite quiver $Q$ with frozen vertices, almost all (weakly) balanced monotone mutation sequences applied to $Q$ eventually become sign-coherent (\Cref{thm: geq 3 eventually sc}). We achieve this by combining the results of the previous section with Warkentin's work showing that most vertices in the exchange graph are forks.

Warkentin showed there is a uniform bound (depending only on rank) on the number of mutations needed to turn any mutation-infinite quiver into a fork.

\begin{lemma}[{\cite[Theorem 5.1]{Warkentin}}]\label{lem: fork in finite mutations}
Let $\Gamma$ be the exchange graph of a mutation-infinite connected quiver of rank $n$.  Then there is a bound $D(n) = O(n^3)$ depending only on $n$ such that any vertex of $\Gamma$ has at most distance $D(n)$ to a fork.    
\end{lemma}

Using a slight modification of Warkentin's approach, we show an analogous result holds for ice forks.  

\begin{prop}\label{prop: add frozen to ice fork}
Let $Q$ be a rank $n \geq 3$ quiver with vertices $V$, such that $Q|_{V \cut \{u\}}$ is an ice fork for some frozen vertex $u$.
Then there is an ice fork at most $\frac{2}{3}(n+1)^3 + 2$ mutations from $Q$. 
\end{prop}
\begin{proof}
First, perform a reduced mutation sequence $\mathbf M_1$ of $|\mathbf M_1| = \frac{1}{3}(n+1)^3 + 1$ mutations to $Q$ whose first mutation is not at the point of return of $Q|_{V\cut \{u\}}$.  
By \Cref{lem:tree lemma}, all quivers of mutation distance at most $\frac{1}{3}(n+1)^3$ from $(Q|_{V \cut \{u\}})_{\mathbf M_1}$ are ice forks.  
Then by \cite[Corollary 3.10]{Warkentin} (setting $k = u$), there is a mutation sequence $\mathbf M_2$ of at most $\frac{1}{3}(n+1)^3$ mutations such that $(Q|_{[n] \cup \{u\}})_{\mathbf{M}_1 \mathbf M_2}$ is an ice fork.  Moreover, since $|\mathbf M_2| \leq |\mathbf M_1|$, $(Q|_{V \cut \{u\}})_{\mathbf{M}_1 \mathbf{M}_2}$ is still an ice fork (where we've used $\mathbf{M}_1 \mathbf{M}_2$ for the concatenation of the two mutation sequences, first performing $\mathbf{M}_1$ and then $\mathbf{M}_2$).  
Then perform a mutation at a vertex $r$ that is not the point of return of either $(Q|_{V \cut \{u\}})_{\mathbf{M}_1 \mathbf{M}_2}$ or $(Q|_{[n] \cup \{u\}})_{\mathbf{M}_1 \mathbf{M}_2}$. The resulting quiver $Q_{\mathbf{M}_1\mathbf{M}_2r}$ is an ice fork (with point of return $r$), as desired.
\end{proof}

\begin{rem}
The rank $n \geq 3$ condition in \Cref{prop: add frozen to ice fork} is necessary. See \Cref{eg: ice forks in rank 2}.
\end{rem}

\begin{lemma}
\label{lem: ice fork in finite mutations}
Let $\Gamma$ be the exchange graph of  connected quiver $Q$ of rank $n$ with $m$ frozen vertices such that $Q^{\mut}$ is mutation infinite.  Then there is a bound $D(m,n)$ depending only on $n$ and $m$ such that any vertex of $\Gamma$ is at most $D(m,n)$ mutations from an ice fork.  
\end{lemma}
\begin{proof}
By \Cref{lem: fork in finite mutations}, we can mutate $Q^{\mut}$ to a fork with at most $D(n)$ mutations. 
So we may assume $Q^{\mut}$ is a fork.
We argue by induction on the number of frozen vertices $m$. 
When $m=1$, the claim follows from \cite[Corollary 3.10]{Warkentin} (setting $k = u_1$): there is a mutation sequence of at most $\frac{1}{3}(n+1)^3$ mutations taking any quiver with a single frozen vertex to an ice fork.
Now suppose the claim is true for $m$ frozen vertices. Let $V$ be the set of vertices of $Q$ and pick a frozen vertex $u$.
By the inductive hypothesis, we can mutate $Q|_{V \cut \{u\}}$ to an ice fork in $D(m,n)$ mutations.
By \Cref{prop: add frozen to ice fork}, we can thus mutate $Q$ to an ice fork in $D(m+1,n) = D(m,n) + \frac 1 3 (n+1)^3 + 2$ mutations.
\end{proof}

\begin{rem}
Unwrapping the induction in \Cref{lem: ice fork in finite mutations}, we find that $$D(m,n) = D(n) + \frac{1}{3}(n+1)^3 + (m-1)\left(\frac{1}{3}(n+1)^3 + 2\right).$$ We will not need the precise value.
\end{rem}

We have the following analogs of \cite[Propositions~5.2, 5.4]{Warkentin}. 
Both are immediate consequences of \Cref{lem: ice fork in finite mutations}.
\begin{prop}
\label{prop: ice forks get wandered into}
Suppose $Q$ is a quiver with at least one frozen vertex and $Q^{\mut}$ is mutation infinite.
A simple random walk on the mutation graph of $Q$ will almost certainly leave the `ice-forkless' part and never return.
%Form an infinite mutation sequence $\mathbf{M} = m_1m_2\cdots$ by choosing $m_i$ uniformly from the mutable vertices of $Q$.
%Then with probability one, there is some $i$ such that $Q^{(j)}_{\mathbf {M}}$ is an ice fork for $j>i$.
\end{prop}

\begin{prop}
\label{prop: ice forks are all the class}
Fix a quiver $Q$ with at least one frozen vertex such that $Q^{\mut}$ is mutation infinite.
Let $N(Q,d)$ denote the number of quivers which are mutation distance at most $d$ from $Q$, and let $F(Q,d)$ denote the number of ice forks which are mutation distance at most $d$ from $Q$.
Then 
$$\lim_{d \rightarrow \infty} \frac{F(Q,d)}{N(Q,d)} = 1.$$
\end{prop}

This leads us to our main theorem (cf. \Cref{thm: intro main}).

%The condition that the mutation sequence is monotone implies that once a fork is reached, the subsequent mutations must all be ascents. 

\begin{thm}\label{thm: geq 3 eventually sc}
Let $Q$ be a quiver with at least one frozen vertex such that $Q^{\mut}$ is mutation-infinite.  For all $i \in \N$, let $m_i$ be chosen uniformly at random from $[n]$.  With probability $1$, the mutation sequence $\mathbf M = m_1m_2\cdots$ on $Q$ is eventually sign-coherent. 
\end{thm}

\begin{proof}
By \Cref{prop: ice forks get wandered into}, this walk almost surely wanders into a single connected component $C$ of ice forks eventually. 
By \Cref{lem:tree lemma}, this subgraph is a rooted $n$-regular tree.
As $Q^{\mut}$ is mutation infinite, the rank of $Q$ is at least $2$.
So $C$ is \emph{transient}, that is, the mutation sequence almost surely returns to each quiver at most finitely many times. 
So we may define the infinite reduced mutation sequence $\mathbf M'$  obtained from repeatedly canceling all $ii$ adjacent pairs in $\mathbf M$.
By construction, $\mathbf M'$ is almost surely balanced (and weakly balanced). The claim follows from \Cref{thm:Asym Sign Coherence}.
\end{proof}

This gives us an entirely combinatorial proof that the classical sign coherence phenomenon holds for almost all of the mutation graph.

\begin{cor}
\label{cor:principal framing result}
Let $Q$ be a principally framed quiver such that $Q^{\mut}$ is mutation infinite quiver. 
Then the proportion of quivers within $d$ mutations of $Q$ that are sign-coherent goes to $1$ as $d$ goes to infinity (cf. \Cref{prop: ice forks are all the class}). 
\end{cor}

We can adapt these methods to handle the case where we require our path to be monotone.  

\begin{lem}
Fix a quiver $Q$ with at least one frozen vertex such that $Q^{\mut}$ is mutation infinite.  Let $M(Q,d)$ denote the number of length $d$ monotone paths from $Q$, and let $S(Q,d)$ denote the number of length $d$ monotone paths from $Q$ that terminate in a sign-coherent ice fork.  Then
$$\lim_{d \to \infty} \frac{S(Q,d)}{M(Q,d)} = 1\,.$$
\end{lem}
\begin{proof}
    It follows from \Cref{lem: ice fork in finite mutations} (via a similar argument to obtaining \Cref{prop: ice forks are all the class}) that almost all length $\frac{d}{2}$ monotone paths from $Q$ terminate in an ice fork.  Then by \Cref{thm:Asym Sign Coherence} along with the monotonicity assumption, this ice fork almost certainly becomes monotone in the subsequent $\frac{d}{2}$ mutations. %d/2 is arbitrary.
\end{proof}

%\begin{cor} Let $Q$ be a quiver with at least one frozen vertex such that $Q^{\mut}$ is infinite-mutation-type.  Let $\mathbf M$ be a randomly-chosen monotone and balanced mutation sequence on $Q$ \AB{specify precisely what this means: I think saying choose randomly among non-visited seeds at each step is enough for monotone, but what about balanced?}.  With probability $1$, $\mathbf M$ is eventually sign-coherent. \end{cor} \begin{proof} If $Q^{\mut}$ is infinite-mutation-type, then only finitely many quivers in its mutation-class have all weights $\leq 2$.  So eventually $\mathbf M$ will reach one of these quivers \cS{confused; we don't need this?}.  Then we simply apply \Cref{thm: geq 3 eventually sc}, removing any non-monotone parts of the mutation sequence if necessary.   \end{proof}

%\cS{I think that the proof below, and probably \Cref{prop: ice forks are all the class}, imply that non-balanced sequences are vanishingly rare for any reasonable measure.}

We conclude this section with some examples demonstrating the necessity of certain conditions to our results.

\begin{figure}
    \centering
\begin{tikzpicture}[->, >={Stealth[round]}, node distance=3cm and 4cm, main/.style = {circle, draw, fill=none, minimum size=6mm, inner sep=0pt}, every edge/.append style={thick}, square/.style={regular polygon,regular polygon sides=4}]

\filldraw[black] (-1,0) node[main, fill = white] (1) {\textcolor{black}{$\boldsymbol{1}$}};
\filldraw[black] (1,0) node[main, fill = white] (2) {\textcolor{black}{$\boldsymbol{2}$}};
\filldraw[black] (0,1) node[main, fill = white, square] (u) {\textcolor{black}{$\boldsymbol{u}$}};
\filldraw[black] (0,-1) node[main, fill = white, square] (v) {\textcolor{black}{$\boldsymbol{v}$}};

\path
 (1) edge["$2$"] (2)
 (2) edge["$k$"] (v)
 (v) edge["$k+1$"] (1)
 (2) edge["$k+1$"'] (u)
 (u) edge["$k$"'] (1);
\end{tikzpicture}
    \caption{A family of quivers which are $k$ mutations from an ice fork.}
    \label{fig: 2 vertex far from ice}
\end{figure}
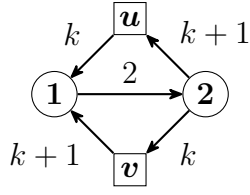

\begin{exmp}
\label{eg: ice forks in rank 2}
Fix an integer $k>0$ and let $Q$ be the corresponding quiver shown in Figure~\ref{fig: 2 vertex far from ice}. 
Then $Q$ is at least $k$ mutations from an ice fork, even though both $Q|_{12u}$ and $Q|_{12v}$ are ice forks for $k>2$, they have different points of return. Thus there is no uniform bound on the number of mutations required to make a rank $2$ quiver into an ice fork.
Therefore the $n \geq 3$ condition in \Cref{prop: add frozen to ice fork} is necessary. 
\end{exmp}

The following example shows that eventual sign coherence requires some condition on the mutation sequence that is stronger than reducedness, such as simplicity or monotonicity. 

\begin{exmp}
Consider the quiver with mutable part $1 \stackrel{2}{\rightarrow} 2 \stackrel{1}{\rightarrow} 3 \stackrel{1}{\rightarrow} 4$ %1 \Rightarrow 2 \rightarrow 3 \rightarrow 4$ 
with two frozen vertices $u,v$ and arrows:
$2 \stackrel{2}{\rightarrow} u \stackrel{2}{\rightarrow} 1$ and $u \stackrel{1}{\rightarrow} 4 \stackrel{1}{\rightarrow} v.$ %$2 \Rightarrow u \Rightarrow 1$ and $u \rightarrow 4 \rightarrow v$.
Note that the vertex $4$ is not sign-coherent.  This quiver is on a balanced mutation cycle ($\overline{43434213434312}$) with a vertex that is not sign-coherent in infinitely many quivers.  
%\cS{this example shows that we need some condition on the mutation sequence, like simplicity, in general}
%(somewhat unsatisfying in that the mutations at 12 are really cancelling over an isomorphism)
\end{exmp}

\section{Brog Quivers}
\label{sec:brog}

This section is focused on the introduction of a new class of \emph{brog} quivers, which allow us to extend the ideas of \Cref{sec:fork eventual coherence} to the mutation-finite setting.  Brog quivers arise after one nice mutation, and, like ice forks, quickly become sign-coherent after mutation at each mutable vertex.  We use this framework in \Cref{sec:low rank} to combinatorially prove eventual sign coherence for mutation-abundant rank $3$ quivers.

\begin{definition}
    Two mutable vertices $i$ and $j$ of a quiver are called \emph{complementary} if for any frozen vertex $u$, whenever the arrow weights $b_{iu}$ and $b_{ju}$ are non-zero then they have opposite signs.
\end{definition}

Note that by definition, any red vertex is complementary to any green vertex.

\begin{defn}
A quiver is called \emph{brog} if we may color the vertices blue, red, orange and green so that:
\begin{enumerate}
\item the vertices colored red are precisely the red vertices in the sense of Definition~\ref{defn: red green}; 
\item the vertices colored green are precisely the green vertices in the sense of Definition~\ref{defn: red green}; 
\item each blue vertex $i$ satisfies $b_{ij} \geq 0$ whenever $j$ is red and $b_{ij} \leq 0$ whenever $j$ is green;
\item each orange vertex $i$ satisfies $b_{ij} \geq 0$ whenever $j$ is green and $b_{ij} \leq 0$ whenever $j$ is red;
\item each pair of a blue vertex and an orange vertex are complementary.
\end{enumerate}
\end{defn}

Essentially, saying that a quiver is brog means that the arrows between colorful vertices are as in the following diagram, where we do not specify the arrow directions between the complementary pairs (red/green and blue/orange).

\begin{figure}[h]
\begin{tikzpicture}[->, >={Stealth[round]}, node distance=3cm and 4cm, main/.style = {circle, draw, fill=none, minimum size=6mm, inner sep=0pt}, every edge/.append style={thick}, square/.style={regular polygon,regular polygon sides=4}]

\node[main, fill = blue] (1) {\textcolor{white}{$\boldsymbol{b}$}};
\node[main, fill = red] (2) [below right=1cm of 1] {\textcolor{white}{$\boldsymbol{r}$}};
\node[main, fill = darkspringgreen] (4) [below left =1cm of 1] {\textcolor{white}{$\boldsymbol{g}$}};
\node[main, fill = orange] (3) [below left =1cm of 2] {\textcolor{white}{$\boldsymbol{o}$}};

%\node[draw] (Q) [below left=2cm of 1]{$Q=$}

\path 
    (1) edge (2)
    (2) edge (3)
    (3) edge (4)
    (4) edge (1);
%\path (1)+(0,-0.35) edge[dashed,-] (3);
%\path (2)+(-0.35,0) edge[dashed,-] (4);
\end{tikzpicture}
\centering
\caption{A schematic depicting the prescribed arrow directions in a brog quiver.}
\label{fig: brog order}
\end{figure}
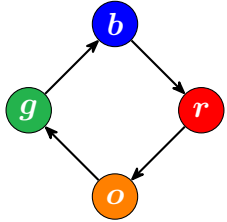
  The term \emph{brog} comes from reading the first letter of each color according to this ordering.

\begin{rem}
    Another way of remembering the coloring is that a vertex is blue if freezing it preserves the set of red and green vertices, and it is orange if freezing results in having no red or green vertices.
\end{rem}

\begin{obs}
Note that the property of a quiver being brog is \emph{hereditary}, that is, it is preserved under taking subquivers.
\end{obs}

Though the brog conditions may seem strong, it turns out that many quivers are brog. In particular, ice forks with two frozen vertices are brog.

\begin{prop}
\label{prop: ice forks brog}
Ice forks with two frozen vertices are brog quivers.
\end{prop}

\begin{proof}
Let $r$ be the point of return and $u,v$ the frozen vertices of an ice fork $Q$.
The color of the sign-coherent vertices are determined by the orientations of the arrows. 
Except possibly for $r$ (which we will treat separately), we color all other vertices blue.

Fix a blue vertex $i \neq r$ and, without loss of generality, say $u \rightarrow i \rightarrow v$. 
Now fix a sign-coherent vertex $j\neq r$, we know $Q|_{iju}$ and $Q|_{ijv}$ are acyclic as $r\not \in \{i,j,u,v\}$.
If $j$ is green, then $j \rightarrow u \rightarrow i$, and so $j \rightarrow i$. 
Similarly, if $j$ is red then $i \rightarrow v \rightarrow j$, and so $i \rightarrow j$. 
Finally, if $j\neq r$ is another blue vertex, then if $u \leftarrow j \leftarrow v$ we would have an oriented $4$-cycle through $u,i,v,j$, again contradicting that $r \not \in \{i,j,u,v\}$; thus all blue vertices have the same orientations $u \rightarrow i \rightarrow v$.

Now let us consider the point of return $r$.
We will repeatedly use that $r$ is not the elbow (\Cref{def:elbow-oriented}) of an acyclic subquiver of $Q$.
Suppose $r$ is sign-coherent and fix a blue vertex $i$.
If $r$ is red, $i \leftarrow u \rightarrow r$, and so $i \rightarrow r$.
If $r$ is green, $i \rightarrow v \leftarrow r$, and so $r \rightarrow i$.

If $r$ is not sign-coherent then color $r$ orange. We have many orientations to check in this case.
Fix another vertex $j$.
If $j$ is blue and $j \rightarrow r$ then as $Q|_{jrv}$ is acyclic we have $v \rightarrow r$. Similarly if $j \leftarrow r$ then $Q|_{jru}$ is acyclic and $u \leftarrow r$. In either case, $u \rightarrow r \rightarrow v$ and $j$ and $r$ are complementary. 
(If there are no blue vertices we may relabel $u,v$ so that $u \rightarrow r \rightarrow v$.)
If $j$ is red, then $j \leftarrow u \rightarrow r$, so $Q|_{jru}$ is acyclic and $j \rightarrow r$.
If $j$ is green, then $r \rightarrow v \leftarrow j$, again $Q|_{jrv}$ is acyclic and thus $r \rightarrow j$.
\end{proof}

\begin{figure}
\centering

\begin{tikzpicture}[->, >={Stealth[round]}, node distance=3cm and 4cm, main/.style = {circle, draw, fill=none, minimum size=6mm, inner sep=0pt}, every edge/.append style={thick}, square/.style={regular polygon,regular polygon sides=4}]

\draw[black, dashed, decoration={markings, mark=at position 0.08 with {\arrow{<}}}, postaction={decorate}, opacity=0.5] (0,0) circle (2cm);

\filldraw[black] (0,0)++(0:2cm) node[main, fill = darkspringgreen] (1) {\textcolor{white}{$\boldsymbol{4}$}};
\filldraw[black] (0,0)++(45:2cm) node[main, fill = darkspringgreen] (2) {\textcolor{white}{$\boldsymbol{2}$}};
\filldraw[black] (0,0)++(90:2cm) node[main, fill = white] (3) {\textcolor{black}{$\boldsymbol{3}$}};
\filldraw[black] (0,0)++(135:2cm) node[main, fill = red] (4) {\textcolor{white}{$\boldsymbol{5}$}};
\filldraw[black] (0,0)++(180:2cm) node[main, fill = red] (5) {\textcolor{white}{$\boldsymbol{1}$}};
\filldraw[black] (0,0)++(225:2cm) node[main, fill = white, square] (6) {\textcolor{black}{$\boldsymbol{6}$}};
\filldraw[black] (0,0)++(270:2cm) node[main, fill = darkspringgreen] (7) {\textcolor{white}{$\boldsymbol{8}$}};
\filldraw[black] (0,0)++(315:2cm) node[main, fill = darkspringgreen] (8) {\textcolor{white}{$\boldsymbol{7}$}};

\path
 (1) edge (8)
 (8) edge (7)
 (7) edge (6)
 (6) edge (5)
 (5) edge (4)
 (4) edge (3)
 (3) edge (2)
 (2) edge (1)
 (1) edge (7)
 (8) edge (6)
 (7) edge (5)
 (6) edge (4)
 (5) edge (3)
 (2) edge (4)
 (3) edge (1)
 (2) edge (7)
 (2) edge (8)
 (1) edge (5)
 (1) edge (4)
 (1) edge (6)
 (2) edge (6)
 (2) edge (5)
 (3) edge (8)
 (3) edge (7)
 (8) edge (5)
 (8) edge (4)
 (7) edge (5)
 (7) edge (4)
 (6) edge[dashed,-] (3);
\end{tikzpicture}
\caption{Either orientation of the dashed edge between $3$ and $6$ gives a possible orientation of a fork, drawn with the vertices placed onto a circle according to the `canonical' cyclic ordering. Vertex $6$ is frozen, and vertex $3$ is necessarily the point of return (as it is the only vertex contained in every oriented $3$-cycle). % Weights are not shown.
}
\label{fig: ice fork coloring}
\end{figure}
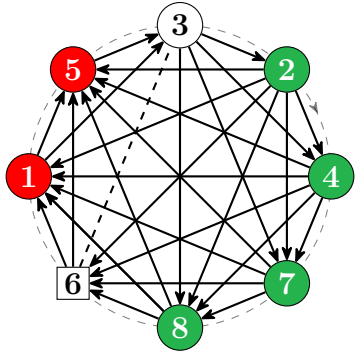

\begin{rem}
\label{rem: ice forks and coloring}
\Cref{prop: ice forks brog} is connected to several structural results for forks, though our proof obfuscates them somewhat. 
In both \cite[Proposition 15.10]{COQ} and \cite[Lemma 5.1, Corollary 5.2]{EJLN} a `canonical' cyclic ordering of the vertices in a fork are constructed, both of which agree with the visit order coming from a particular (oriented) Hamiltonian cycle of the vertices. 
In a fork with just one frozen vertex $i$ and $r$ the point of return, these two vertices split the cyclic order into two intervals (see the diagram in \Cref{fig: ice fork coloring}).
Only red vertices appear in the open interval from $i$ to $r$, while only green vertices appear in the open interval from $r$ to $i$ (the point of return $r$ can be either color). 
By freezing a second vertex, we partition this ordering into at most $4$ colors. For example, freezing $7$ in \Cref{fig: ice fork coloring} would create a blue vertex $8$ and make vertex $3$ a green or orange vertex (depending on the orientation of arrows between $3$ and $6$).
\end{rem}

It turns out that under just one `nice' mutation, a quiver can be made brog.  

\begin{thm}\label{thm: brog}
Let $Q$ be a complete quiver with $2$ frozen vertices.  If $j$ is cycle-preserving and not the apex of a vortex, then $\mu_j(Q)$ is brog. 
\end{thm}

\begin{proof}
    It is enough to show that subquivers of the forms shown in \Cref{fig: forbidden brog subquivers} cannot appear in $ \mu_j(Q)$.  

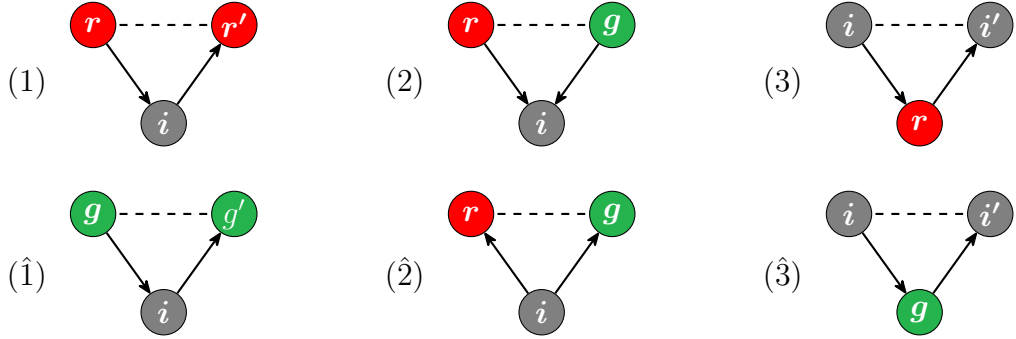
\begin{figure}[h]
\begin{tikzpicture}[->, >={Stealth[round]}, node distance=3cm and 4cm, main/.style = {circle, draw, fill=none, minimum size=6mm, inner sep=0pt}, every edge/.append style={thick}, square/.style={regular polygon,regular polygon sides=4}]

\node[main, fill = red] (1) {\textcolor{white}{$\boldsymbol{r}$}};
\node[main, fill = gray] (3) [below right = 0.866cm and 0.5cm of 1]{\textcolor{white}{$\boldsymbol{i}$}};
\node[main, fill = red] (2) [above right = 0.866cm and 0.5cm  of 3]{\textcolor{white}{$\boldsymbol{r'}$}};
\node(4) [below left = 0.2cm and 0.25cm  of 1]{\textcolor{black}{$(1)$}};

\path (1)+(0.35,0) edge[dashed,-] (2);
\path (1) edge (3);
\path (3) edge (2);

\begin{scope}[shift = {(0,-2.5)}]
\node[main, fill = darkspringgreen] (1) {\textcolor{white}{$\boldsymbol{g}$}};
\node[main, fill = gray] (3) [below right = 0.866cm and 0.5cm of 1]{\textcolor{white}{$\boldsymbol{i}$}};
\node[main, fill = darkspringgreen] (2) [above right = 0.866cm and 0.5cm  of 3]{\textcolor{white}{$g'$}};
\node(4) [below left = 0.2cm and 0.25cm  of 1]{\textcolor{black}{($\hat{1}$)}};

\path (1)+(0.35,0) edge[dashed,-] (2);
\path (1) edge (3);
\path (3) edge (2);
\end{scope}

\begin{scope}[shift = {(5,0)}]
\node[main, fill = red] (1) {\textcolor{white}{$\boldsymbol{r}$}};
\node[main, fill = gray] (3) [below right = 0.866cm and 0.5cm of 1]{\textcolor{white}{$\boldsymbol{i}$}};
\node[main, fill = darkspringgreen] (2) [above right = 0.866cm and 0.5cm  of 3]{\textcolor{white}{$\boldsymbol{g}$}};
\node(4) [below left = 0.2cm and 0.25cm  of 1]{\textcolor{black}{$(2)$}};

\path (1)+(0.35,0) edge[dashed,-] (2);
\path (1) edge (3);
\path (2) edge (3);
\end{scope}

\begin{scope}[shift = {(5,-2.5)}]
\node[main, fill = red] (1) {\textcolor{white}{$\boldsymbol{r}$}};
\node[main, fill = gray] (3) [below right = 0.866cm and 0.5cm of 1]{\textcolor{white}{$\boldsymbol{i}$}};
\node[main, fill = darkspringgreen] (2) [above right = 0.866cm and 0.5cm  of 3]{\textcolor{white}{$\boldsymbol{g}$}};
\node(4) [below left = 0.2cm and 0.25cm  of 1]{\textcolor{black}{$(\hat{2})$}};

\path (1)+(0.35,0) edge[dashed,-] (2);
\path (3) edge (1);
\path (3) edge (2);
\end{scope}

\begin{scope}[shift = {(10,0)}]
\node[main, fill = gray] (1) {\textcolor{white}{$\boldsymbol{i}$}};
\node[main, fill = red] (3) [below right = 0.866cm and 0.5cm of 1]{\textcolor{white}{$\boldsymbol{r}$}};
\node[main, fill = gray] (2) [above right = 0.866cm and 0.5cm  of 3]{\textcolor{white}{$\boldsymbol{i'}$}};
\node(4) [below left = 0.2cm and 0.25cm  of 1]{\textcolor{black}{$(3)$}};

\path (1)+(0.35,0) edge[dashed,-] (2);
\path (1) edge (3);
\path (3) edge (2);
\end{scope}

\begin{scope}[shift = {(10,-2.5)}]
\node[main, fill = gray] (1) {\textcolor{white}{$\boldsymbol{i}$}};
\node[main, fill = darkspringgreen] (3) [below right = 0.866cm and 0.5cm of 1]{\textcolor{white}{$\boldsymbol{g}$}};
\node[main, fill = gray] (2) [above right = 0.866cm and 0.5cm  of 3]{\textcolor{white}{$\boldsymbol{i'}$}};
\node(4) [below left = 0.2cm and 0.25cm  of 1]{\textcolor{black}{$(\hat{3})$}};

\path (1)+(0.35,0) edge[dashed,-] (2);
\path (1) edge (3);
\path (3) edge (2);
\end{scope}
\end{tikzpicture}
\caption{
These six forms of subquivers are the minimal configurations violating the brog conditions.  The grey vertices (with label $i$ or $i'$) are assumed to be mutable vertices that are neither red nor green, and moreover $i$ and $i'$ are not complementary.}
\label{fig: forbidden brog subquivers}
\end{figure}

Denote the frozen vertices by $u,v$.  Without loss of generality, we assume $u \rightarrow i \rightarrow v$ and $u \rightarrow i' \rightarrow v$.  Note that, for $c = 1,2,3$, Case $(\hat{c})$ follows from Case $(c)$ by a global reversal of arrows and relabeling of vertices.  So it is enough to handle Cases $(1)$, $(2)$, and $(3)$.

We first show that $j$ cannot be one of the vertices shown in the subquivers of forms $(1)$, $(2)$, and $(3)$.

\emph{Case (1):} Note that $r$ is the elbow of the acyclic quiver $\mu_j(Q)|_{uri}$, so we cannot have $j = r$.  Similarly, $i$ is the elbow of $\mu_j(Q)|_{uir'}$, so we cannot have $j = i$.  If $j = r'$, we must have $r \rightarrow r'$ lest $\mu_j(Q)|_{ur'r}$ is acyclic with elbow $r'$.  However, then $r'$ is the apex of the vortex quiver $\mu_j(Q)|_{r'riv}$.  

\emph{Case (2):}  Since $r$ is the elbow of $\mu_j(Q)|_{uri}$ and $i$ is the elbow of $\mu_j(Q)|_{giv}$, we have $r \neq j \neq i$.  If $j = g$, we must have $r \leftarrow g$ lest $\mu_j(Q)|_{rgi}$ is acyclic with elbow $g$.  However, then $g$ is the apex of the vortex quiver $\mu_j(Q)|_{griv}$.

\emph{Case (3):} Since $i$ is the elbow of $\mu_j(Q)|_{uir}$ and $r$ is the elbow of $\mu_j(Q)|_{uri'}$, we have $i \neq j \neq r$.  Then no matter the orientation of the arrow between $i$ and $i'$, the vertex $i'$ will be an elbow in either $\mu_j(Q)|_{ui'i}$ or $\mu_j(Q)|_{ii'v}$.  So we have $j \neq i'$.

Now we show that even if $j$ is not among a subquiver of the six forms shown above, these subquivers still cannot arise in $\mu_j(Q)$.  Note that if $j$ is red or green and a forbidden rank $3$ subquiver exists in $\mu_j(Q)$, then a forbidden rank $3$ subquiver that includes $j$ exists in $\mu_j(Q)$.  So we can assume that $j$ is neither red nor green.  An analogous statement holds if $j$ is not complementary to $i$, so we can assume $v \rightarrow j \rightarrow u$.

We observe that if a vertex $w$ is red, we must have $w \rightarrow j$ lest $\mu_j(Q)|_{vjw}$ is acyclic with elbow $j$.  Similarly if $w$ is green, we must have $w \leftarrow j$ lest $\mu_j(Q)|_{wju}$ is acyclic with elbow $j$.  

\emph{Cases (1) and (2):} We must have $i \rightarrow j$ lest $\mu_j(Q)|_{rji}$ is acyclic with elbow $j$.   But then $j$ is the apex of a vortex in $\mu_j(Q)|_{jriv}$.

\emph{Case (3):} We must have $i' \rightarrow j$ lest $\mu_j(Q)|_{rji'}$ is acyclic with elbow $j$.  But then $j$ is the apex of a vortex in $\mu_j(Q)|_{jri'v}$.\end{proof}

\begin{rem}
Note that if a brog quiver has no orange (resp. blue) vertices, it is possible for it to have a pair of complementary blue (resp. orange) vertices.  However, this phenomenon does not occur in $\mu_j(Q)$ under the conditions of \Cref{thm: brog}.  This can be seen, for example, by showing that $Q$ is a subquiver of some $\widetilde Q$ such that $\mu_j(\widetilde Q)$ has both blue and orange vertices.
\end{rem}

\begin{rem}
The condition that $Q$ has exactly $2$ frozen vertices is necessary in both \Cref{prop: ice forks brog} and \Cref{thm: brog}. 
When there are $3$ or more frozen vertices it is possible to have $3$ mutable vertices which are not sign-coherent and no pair of which is complementary. These will disappear when we restrict to any pair of frozen vertices, some will become sign-coherent or they will split into complementary sets.  Fortunately this is not important for eventual sign coherence, a vertex that is either red or green in each subquiver with 2 frozen vertices will be red or green in the quiver. % (as will be used in the proof of \Cref{thm: brog coherence}).
\end{rem}

In a sequence of cycle-preserving mutations (as in \Cref{thm: brog}) on a complete quiver, \Cref{prop: cycle preserving is like a por} shows we only need to check that we are not mutating at the apex of a vortex in the first mutation.

\begin{cor}
\label{cor: brog forever}
Suppose $Q$ is a brog quiver with exactly two frozen vertices and $\mathbf M = m_1 m_2\cdots$ is a reduced cycle-preserving mutation sequence with $m_1$ not the apex of a vortex.
Then $Q_{\mathbf M}^{(j)}$ is brog for all $j$ and vortex-free for $j>0$.
\end{cor}

\begin{proof}
This follows by induction on $j$ using \Cref{prop: cycle preserving is like a por} and \Cref{thm: brog}. 
\end{proof}

%\cS{I think this is true (modulo some trivialities) without the 2-frozen condition, but we don't prove it; \Cref{thm: brog} generalizes to any number of frozen vertices if $Q$ is brog, in that proof we only need it to go from $i,i'$ not complementary $\implies$ same orientations. We only use the 2 vertex version, probably best to keep it like this.}

\subsection{Eventual Sign Coherence for Brog Quivers}

We now prove the following generalization of \Cref{thm:Asym Sign Coherence}.
This will follow from a similar argument. 

\begin{thm}
\label{thm: brog coherence}
Suppose that $Q$ is a complete quiver and $\mathbf M = m_1 m_2 \cdots$ is a reduced, weakly balanced, and cycle-preserving mutation sequence with $m_1$ not the apex of a vortex in~$Q$.
Then $Q$ is eventually sign-coherent on $\mathbf M$.
\end{thm}

\begin{lem}
\label{lem: brog colors stay}
Let $Q$ be a complete brog quiver with $j$ a cycle-preserving blue vertex.
Let~$k$ be another mutable vertex. 
If $k$ is red or green in $Q$ then it is also red or green in~$\mu_j(Q)$.
\end{lem}

The proof is essentially identical to \Cref{lem: colors stay}, we omit it.
%\begin{proof} There is nothing to show if $j$ is sign-coherent. Assumptions exactly imply that $k$ is a source/sink in the $u,v,j,k$ subquiver, so we don't change it. \end{proof}

\begin{lem}
\label{lem: orange away}
Let $Q$ be a complete brog quiver and $j$ be both cycle-preserving and not the apex of a vortex.
If $j$ is not sign-coherent, then we may color $Q' = \mu_{j}(Q)$ so that only $j$ is orange.
\end{lem}

(It is possible that there are no sign-coherent vertices in $Q'$, in which case there is a second coloring where $j$ is blue and all other vertices are orange.) %this is silly - don't pick that coloring. But currently I can't stop you.

\begin{proof}
By \Cref{thm: brog} the quiver $Q'$ is brog. 
So it suffices to show that for any green (resp. red) vertex $k$ will have $k \leftarrow j$ (resp. $k \rightarrow j$) in $Q'$, and all other vertices are complementary to $j$ in $Q'$.
Without loss of generality, say $u,v$ are frozen with $u \rightarrow j \rightarrow v$ in $Q$ (and thus $u \leftarrow j \leftarrow v$ in $Q'$).

First, suppose $k$ is green in $Q$. If $j$ is blue in $Q$, then $k$ is green and $j \rightarrow k$ in $Q'$. 
If instead~$j$ is orange then we know the following orientations in $Q$:
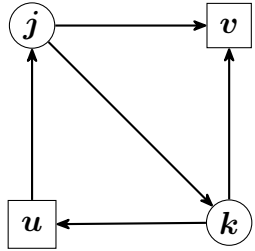
\begin{center}
        
\begin{tikzpicture}[->, >={Stealth[round]}, node distance=3cm and 4cm, main/.style = {circle, draw, fill=none, minimum size=6mm, inner sep=0pt}, every edge/.append style={thick}, square/.style={regular polygon,regular polygon sides=4}]

\node[main] (j) {$\boldsymbol{j}$};
\node[draw, square, inner sep=.8mm] (v) [right=2cm of j] {$\boldsymbol{v}$};
\node[main] (k) [below=2cm of v] {$\boldsymbol{k}$};
\node[draw, square, inner sep=.8mm] (u) [below=2cm of j] {$\boldsymbol{u}$};

\path 
    (u) edge (j)
    (j) edge (v)
    (j) edge (k)
    (k) edge (u)
    (k) edge (v);
\end{tikzpicture}
%\label{fig:orientations with j uncolored}

\end{center}
Because $j$ is cycle-preserving, in $Q'$ we have $u \rightarrow k \rightarrow v$, so $k$ is complementary to $j$ in~$Q'$. Reversing all arrows gives the same result when $k$ is red.

Suppose $k$ is complementary to $j$ in $Q$. 
If $k \leftarrow j$ in $Q$ then $k$ is red with $k \rightarrow j$ in~$Q'$.
Similarly, if $k \rightarrow j$ in $Q$ then $k$ is green with $k \leftarrow j$ in~$Q'$.
In either case, this implies $j$ is orange in $Q'$.

Finally, if $u \rightarrow k \rightarrow v$ in $Q$ then we still have $u \rightarrow k \rightarrow v$ in $Q'$. % (we can only add arrows in the same directions, no oriented 2-cycles are formed in the mutation process).
So $k$ is complementary to~$j$ in $Q'$.
\end{proof}

\begin{rem}
\Cref{cor: brog forever} and \Cref{lem: orange away} show ways that brog quivers under cycle-preserving mutations behave like forks, with orange vertices acting like the point of return. See also \Cref{rem: ice forks and coloring}.
\end{rem}

\begin{lem}
\label{lem: orange is rare}
Let $Q$ be a complete brog quiver and $\mathbf M=m_1m_2\cdots$ a reduced cycle-preserving mutation sequence with $m_1$ not the apex of a vortex.
Then there are at most two $j$ such that~$m_j$ is an orange vertex in~$Q_{\mathbf M}^{(j-1)}$.
\end{lem} %why 2 and not 1? Same technicality as in \Cref{lem: orange away}

\begin{proof}
%Without loss of generality, $m_1$ is orange in $Q = Q^{(0)}$.
By \Cref{lem: orange away}, it suffices to show that mutations at sign-coherent vertices do not introduce new orange vertices.
So let $j=m_i$ be sign-coherent in~$Q_{\mathbf M}^{(i-1)}$ and fix another mutable vertex $k$; we will show $k$ is not orange in~$Q_{\mathbf M}^{(i)}$. 
By \Cref{lem: skip colored}, we reduce to the case where $k$ is not sign-coherent.
Without loss of generality, say $j$ is green in~$Q_{\mathbf M}^{(i-1)}$ and pick frozen vertices so that $u \rightarrow k \rightarrow v$.
If $j \rightarrow k$ in~$Q_{\mathbf M}^{(i-1)}$ then $j$ is red in $Q_{\mathbf M}^{(i)}$ with $j \leftarrow k$, so $k$ is not orange in $Q_{\mathbf M}^{(i)}$. %indeed, k is blue. but we needn't look at the other arrows in this case.
If $j \leftarrow k$ in~$Q_{\mathbf M}^{(i-1)}$ then $k$ is green in $Q_{\mathbf M}^{(i)}$.
\end{proof}

\begin{proof}[{(Proof of \Cref{thm: brog coherence})}]
It suffices to prove the claim when $Q$ has exactly two frozen vertices, as a quiver is sign-coherent whenever all subquivers with $2$ frozen vertices are. 
By \Cref{cor: brog forever}, each $Q_{\mathbf M}^{(i)}$ is brog and vortex-free for $i >0$, thus each $m_i$ is cycle-preserving and not the apex of a vortex in $Q_{\mathbf M}^{(i-1)}$.
By \Cref{lem: orange is rare}, there is some $k$ such that all mutations $m_j$ are not at orange vertices for $j > k$. Let $R = Q^{(k)}$ and $\mathbf M' = m_{k+1} m_{k+2} \cdots$, so all mutations in $\mathbf M'$ are at blue, green, or red vertices.
The set of sign-coherent vertices in $R_{\mathbf M'}^{(i)}$ only increases with $i$ by \Cref{lem: skip colored} and \Cref{lem: brog colors stay}.
The sequence $\mathbf M'$ is weakly balanced because $\mathbf M$ is. 
So there is some $\ell$ such that each mutable vertex appears in $\{m_{k+1},m_{k+2},\ldots ,m_{k+\ell}\}$. 
By \Cref{lem:mutation makes more color}, $R_{\mathbf M'}^{(\ell+1)}$ is sign-coherent (and thus so are all $Q_{\mathbf M}^{(j)}$ for $j > \ell+k$).
\end{proof}

\section{Low Rank Quivers}
\label{sec:low rank}

We now prove the asymptotic sign coherence conjecture for mutation-abundant (\Cref{def:mu-abundant}) rank $3$ quivers.  Previously the only rank $3$ case where the asymptotic sign coherence conjecture was the mutation sequence $\overline{123}$ on the Markov quiver, as shown by Gekhtman-Nakanishi \cite[Section 4]{GekhtmanNakanishi}.  We additionally recover their result that the asymptotic sign coherence conjecture holds for rank $2$ quivers.  We then discuss applications to quivers whose mutation class has finite ice-forkless part.

Note that the cases in which Gekhtman and Nakanishi proved their asymptotic sign coherence conjecture all have mutation-finite mutable part, so the results of \Cref{sec:wander} do not apply. We instead use the brog quivers introduced in \Cref{sec:brog} and ideas from \Cref{sec:fork eventual coherence} to establish sign coherence results in this setting.  

Recall we use $b_{ij}(Q)$ to denote the number of arrows oriented $i \rightarrow j$ minus the number of arrows oriented $j \rightarrow i$ in $Q$.
The observations in \Cref{lem: rank 2 ascents} are not new, but will be useful throughout this section. 

\begin{lem}
\label{lem: rank 2 ascents}
Let $Q$ be a rank $2$ quiver with one frozen vertex $u$, and $b_{12}(Q) \geq 2$.
Suppose that mutation at $1$ increases the number of arrows between $2$ and $u$ (ie. $|b_{2u}(\mu_1(Q))| > |b_{2u}(Q)|$). 
Then: 
\begin{enumerate}
    \item $\mu_1(Q)$ is an oriented cycle,
    \item $|b_{u1}(Q) b_{12}(Q)| > 2 |b_{u2}(Q)|$, and 
    \item mutation at $2$ in $\mu_1(Q)$ increases the number of arrows between $1$ and $u$.
\end{enumerate} % 
%pf: we have 
%$$b_{2u}(\mu_1(Q)) = b_{2u}(Q) + b_{21}(Q)b_{1u}(Q) > |b_{2u}(Q)|.$$
%If $b_{2u}(Q) > 0$, then $b_{2u}(\mu_1(Q)) > b_{u1}(\mu_1(Q)) (= b_{1u}(Q))$ from the equality.
%If instead we have $b_{2u}(Q) < 0$ then 
%$$b_{21}(Q)b_{1u}(Q) = b_{2u}(\mu_1(Q)) + b_{u2}(Q) < 2 b_{2u}(\mu_1(Q)).$$
%As $Q^{mut}$ is mutation cyclic, $2 \leq b_{21}(Q)$.
%Therefore again we see $b_{2u}(\mu_1(Q)) > b_{u1}(\mu_1(Q))$.
%In either case, mutation at $2$ increases the number of arrows between $u$ and $1$.
%d = ac-b > b > 0, a >= 2
%d + b = ac > 2c
% if c is less than the average of d,b and d > b, then d > c.
If instead mutation at $1$ preserves the number of arrows (ie. $|b_{2u}(\mu_1(Q))| = |b_{2u}(Q)|$) then either:
\begin{enumerate}
    \item[(a)] $1$ is a source/sink, or
    \item[(b)] $Q$ is an oriented cycle and $|b_{u1}(Q)b_{12}(Q)| = 2|b_{2u}(Q)|$.
\end{enumerate}
In either case, mutation at $2$ in $\mu_1(Q)$ does not decrease the number of arrows between $1$ and~$u$.
\end{lem}

\subsection{Rank 3 Quivers}
\label{subsec:rank3 asymptotic}

\begin{defn}
\label{def:mu-abundant}
A quiver $Q$ is \emph{mutation-abundant} if every weight in every mutation equivalent quiver is at least $2$.
\end{defn}

Our next goal is to prove the following theorem.  

\begin{thm}
\label{thm: mutation cyclic eventually sign coherent}
Let $Q$ be a connected rank $3$ quiver whose mutable part is mutation-abundant, and $\mathbf M$ be a reduced weakly balanced mutation sequence which is cycle-preserving for $Q^{mut}$. %alternative/stronger: such that $Q$ mut i is an abundant oriented cycle for all i.
Then $Q$ is eventually sign-coherent on $\mathbf M$.
\end{thm}

We do this by showing that eventually the quivers on $\mathbf M$ are complete and all mutations are cycle-preserving.  Both properties are established through using ``ascent'' mutations, which increase the number of arrows adjacent to a chosen frozen vertex.  These properties ensure, using results from \Cref{sec:brog}, that the quiver eventually becomes brog and then sign-coherent.  

\begin{defn}
Suppose $Q$ is a $3$-vertex quiver and $u$ is a frozen vertex in $Q$.
We say a mutable vertex $i$ is a \emph{$u$-ascent} if $|b_{ju}(\mu_i(Q))| \geq |b_{ju}(Q)|$ for all mutable vertices $j$, and at least one inequality is strict.
\end{defn}

\begin{lem}
\label{lem: ascents are cycle preserving}
Let $Q$ be a connected rank $3$ quiver with one frozen vertex $u$. 
If $i$ is a $u$-ascent and cycle-preserving on $Q^{mut}$, then $i$ is cycle-preserving.
\end{lem}
\begin{proof}
We show the contrapositive. 
%There are two cases, based on the orientation of $Q^{mut}$.
If $\mu_i(Q)|_{123}$ is acyclic but $Q|_{123}$ is not then $i$ is not cycle-preserving on the mutable part. 
Otherwise there is a $j$ such that $\mu_i(Q)|_{iju}$ is acyclic but $Q|_{iju}$ is not. 
Without loss of generality, say $i=1,j=2$.
So there are arrows between $1$ and $u$.
However, 
$$|b_{2u}(Q)| = |b_{2u}(\mu_1(Q))| + |b_{u1}(\mu_1(Q))b_{12}(\mu_1(Q))| > |b_{2u}(\mu_1(Q))|,$$
so the weight between $u$ and $2$ decreased from $Q$ to $\mu_i(Q)$.
\end{proof}

\begin{lem}
\label{lem: cycle preserving ascents avoid vortices}
Let $Q$ be a rank $3$ quiver with one frozen vertex $u$. 
If $i$ is a $u$-ascent and cycle-preserving, then $\mu_i(Q)$ is not a vortex.
\end{lem}

\begin{proof}
Without loss of generality, $i=1$. 
If $\mu_1(Q)$ is not complete, there is nothing to show.
%As $1$ is a $u$-ascent, it is not a sink or source mutation. So $1$ cannot be the apex of a vortex.
If $1$ is a sink or source in the mutable part of $Q$, then both $\mu_1(Q)|_{12u}$ and $\mu_1(Q)|_{13u}$ are oriented cycles by \Cref{lem: rank 2 ascents} (1), so $\mu_1(Q)$ is not a vortex.
If $1$ is on an oriented path in $Q^{mut}$, say $3 \rightarrow 1 \rightarrow 2$, then $\mu_1(Q)|_{123}$ and (by \Cref{lem: rank 2 ascents} (1)) either $\mu_1(Q)|_{12u}$ or $\mu_1(Q)|_{13u}$ are oriented cycles, so again $\mu_1(Q)$ is not a vortex.
\end{proof}

\begin{lem} 
\label{lem:weakly balanced grows}
Let $Q$ be a connected quiver with a single frozen vertex $u$ and $Q^{mut}$ mutation-abundant, and let $\mathbf M = m_1 m_2 \cdots$ be a weakly balanced mutation sequence of $Q$ which is cycle-preserving on $Q^{mut}$. 
Then for some $i$, the vertex $m_i$ is a $u$-ascent in $Q^{(i-1)}$. 
%Thus there are infinitely many such mutations.
\end{lem}

\begin{proof} 
We first establish that some mutation increases a weight $|b_{iu}|$ adjacent to $u$, then we show that this quickly leads to a $u$-ascent.

There are only finitely many arrows adjacent to the frozen vertex. 
We cannot have infinitely many mutations which decrease the total number of arrows adjacent to $u$ without any increasing the total number of arrows adjacent to $u$; we will eventually run out. 
So we may assume that $Q$ has the least number of arrows $|b_{1u}| + |b_{2u}| + |b_{3u}|$ adjacent to $u$ among all the $Q^{(i)}_{\mathbf M}$. %this is a minimum over a discrete valued sequence with a lower bound; the lower bound is attained.

Without loss of generality, say that $m_1=1$. 
As $\mathbf M$ is weakly balanced, we may further suppose that $1$ is adjacent to $u$ in $Q$. \medskip

\textbf{Case 1.} Suppose $Q^{mut}$ is either an oriented cycle, or it is acyclic and $1$ is the elbow. 
Then without loss of generality, $3 \rightarrow 1 \rightarrow 2$.
By relabeling $2,3$ and reversing all arrows if necessary, we may assume $1$ is red.
Mutation at $1$ adds arrows oriented $u \rightarrow 2$.
%Therefore $|b_{1u}(Q)| = |b_{1u}(Q^{(1)}_{\mathbf M})|$, $|b_{3u}(Q)| = |b_{3u}(Q^{(1)}_{\mathbf M})|$, and 
%\[b_{u2}(Q^{(1)}_{\mathbf M}) = b_{u2}(Q) + b_{u1}(Q) b_{12}(Q).\]
Thus, in order for $|b_{u2}(Q^{(1)}_{\mathbf M})| \geq |b_{u2}(Q)|$, we must have $u \rightarrow 2$ in $Q^{(1)}_{\mathbf M}$. %$b_{u1}(Q) b_{12}(Q) \geq 2 b_{2u}(Q)$

Because $\mathbf M$ is weakly balanced, eventually (after some number of alternating mutations at $1,2$), we mutate at vertex $3$. Say $m_i = 3$ is the first such mutation. 
By \Cref{lem: rank 2 ascents} ((3) or (b)) the subquiver $Q^{(i-1)}_{\mathbf M}|_{12u}$ is an oriented cycle.
Because $(Q^{(1)}_{\mathbf M})^{mut}$ is an oriented cycle and $\mathbf M$ is cycle-preserving on the mutable part, $(Q^{(j)}_{\mathbf M})^{mut}$ is an oriented cycle for $j \geq 1$. So there is at most one oriented path $u \rightarrow i \rightarrow j$, and thus a mutation is a $u$-ascent if it increases any of the weights $\{|b_{1u}|, |b_{2u}|, |b_{3u}|\}$.
If $b_{3u}(Q^{(i-1)}_{\mathbf M}) \neq 0$ then $m_i$ is the elbow of either $Q^{(i-1)}_{\mathbf M}|_{13u}$ or $Q^{(i-1)}_{\mathbf M}|_{23u}$, and thus a $u$-ascent. 
Otherwise $b_{3u}(Q^{(i-1)}_{\mathbf M})=0$. 
Vertex $1$ (resp. $2$) is an elbow in $Q^{(i)}_{\mathbf{M}}|_{13u}$ (resp. $Q^{(i)}_{\mathbf{M}}|_{23u}$). 
Thus, regardless of the value of $m_{i+1}$, the weight $|b_{3u}|$ increases from $Q^{(i)}_{\mathbf M}$ to $Q^{(i+1)}_{\mathbf M}$ and $m_{i+1}$ is a $u$-ascent. \medskip

\textbf{Case 2.} Suppose $1$ is a source or sink in $Q^{mut}$. Reversing arrows if necessary, say $1$ is a source.
If any future mutation $m_i$ is not at a source in $(Q^{(i-1)}_{\mathbf{M}})^{mut}$, then we may replace $Q$ with $Q^{(i)}_{\mathbf{M}}$, and we will eventually find ourselves in Case 1.
We may further assume $u \rightarrow 1$ in $Q$ (if we perform $3$ consecutive source mutations starting at $1$, then $u$ is a source and so $u \rightarrow 1$ in $Q^{(3)}_{\mathbf{M}}$). 

We claim that either $m_1$ or $m_2$ will increase some weight adjacent to $u$. 
If $|b_{2u}(Q^{(1)}_{\mathbf M})| = |b_{2u}(Q)|$ and $|b_{3u}(Q^{(1)}_{\mathbf M})|=|b_{3u}(Q)|$, then $2 \rightarrow u \leftarrow 3$ in $Q$ (we are in Case (b) of \Cref{lem: rank 2 ascents}).
Up to relabeling, say $m_2=2$ and $2 \rightarrow 3$ in $Q^{(1)}_{\mathbf M}$. 
Then $2$ is an elbow in $Q^{(1)}|_{23u}$ and so $|b_{3u}(Q^{(1)})| < |b_{3u}(Q^{(2)})|$. 

To restate: reindexing and relabeling as necessary, we have a mutation $1$ in $Q$ which increases some weight adjacent to $u$, say $b_{2u}$. 
Recall by assumption that all mutations are at a source vertex in the mutable part $(Q^{(i-1)}_{\mathbf{M}})^{mut}$. 
As $Q^{mut}$ is abundant, it is complete. So $\mathbf{M}$ is determined by its first two values.
%Thus we know the following orientations in $Q^{(1)}_{\mathbf M}$:
Without loss of generality, we assume that $1 \rightarrow u \rightarrow 2 \rightarrow 1 \leftarrow 3$ in $Q^{(1)}_{\mathbf{M}}$.

\textbf{Case 2a.} $m_2=3$, and so $\mathbf M = \overline{132}$.

If $3$ is a source in $Q^{(1)}_{\mathbf M}$ then $b_{u3}(Q^{(2)}_{\mathbf M}) \geq 0$ and $Q^{(1)}_{\mathbf M}|_{12u} = Q^{(2)}_{\mathbf M}|_{12u}$.
%Also $b_{u2}(Q^{(2)}_{\mathbf M}) = b_{u2}(Q) + b_{u1}(Q) b_{12}(Q) > |b_{u2}(Q)|$, and so $b_{u2}(Q^{(2)}_{\mathbf M}) > b_{1u}(Q^{(2)}_{\mathbf M})$ (recall that $Q^{mut}$ is abundant). 
Both noticing that the elbow in $Q^{(2)}_{\mathbf M}|_{23u}$ is $2$ and applying \Cref{lem: rank 2 ascents} (3) to $Q^{(2)}_{\mathbf M}|_{12u}$, we see $2$ is a $u$-ascent in $Q^{(2)}_{\mathbf M}$.

If instead $u \rightarrow 3$ in $Q^{(1)}_{\mathbf M}$ then $3$ is an elbow in $Q^{(1)}_{\mathbf M}|_{23u}$ (and so $|b_{2u}(Q^{(1)}_{\mathbf M})| < |b_{2u}(Q^{(2)}_{\mathbf M})| $). 
%If $b_{u3}(Q^{(1)}_{\mathbf M}) b_{31}(Q^{(1)}_{\mathbf M})> 2b_{1u}(Q^{(1)}_{\mathbf M})$ (and thus $|b_{1u}(Q^{(2)}_{\mathbf M})| > |b_{1u}(Q^{(1)}_{\mathbf M})|$) 
If $3$ is a $u$-ascent in $Q^{(1)}_{\mathbf M}$ then we are done. 
Otherwise $|b_{1u}(Q^{(2)}_{\mathbf M})| < |b_{1u}(Q^{(1)}_{\mathbf M})|$, so by \Cref{lem: rank 2 ascents} (2) we have 
\[|b_{1u}(Q^{(2)}_{\mathbf M})| < |b_{1u}(Q^{(1)}_{\mathbf M})| < |b_{2u}(Q^{(1)}_{\mathbf M})| < |b_{2u}(Q^{(2)}_{\mathbf M})| \] 
and $|b_{3u}(Q^{(2)}_{\mathbf M})| < |b_{2u}(Q^{(2)}_{\mathbf M})|$, so $2$ is a $u$-ascent in $Q^{(2)}_{\mathbf M}$.

\textbf{Case 2b.} $m_2=2$, and so $\mathbf M = \overline{123}$.

If $|b_{3u}(Q^{(2)}_{\mathbf M})| > |b_{3u}(Q^{(1)}_{\mathbf M})|$ then $2$ is a $u$-ascent in $Q^{(2)}_{\mathbf M}$. 
So we are left with the case where $2$ increases $|b_{1u}|$ and does not increase $|b_{3u}|$ with the next mutation at $3$, this is exactly the setting of Case 2a after relabeling vertices.
\end{proof}

Just one $u$-ascent mutation effectively implies that all future mutations are.

\begin{lem}
\label{lem: increasing sticky}
\label{lem: increasing sticky for incomplete}
Let $Q$ be a connected rank $3$ quiver with one frozen vertex $u$ such that $Q^{\mut}$ is mutation-abundant. 
Let $\mathbf M = m_1 m_2\cdots $ be a reduced mutation sequence which is cycle-preserving on $Q^{mut}$. 
If $m_1$ is a $u$-ascent in $Q = Q_{\mathbf M}^{(0)}$, 
then each mutation at $m_i$ is either a $u$-ascent in $Q_{\mathbf M}^{(i-1)}$ or not adjacent to $u$.
\end{lem}
\begin{proof}
By induction, it suffices to show that the next vertex $m_i$ for $i > 1$ that is adjacent to $u$ is a $u$-ascent. % some $i>0$ such that $m_i$ increases the number of arrows adjacent to $u$ in $Q_{\mathbf M}^{(i-1)}$.
Without loss of generality, let $m_1=1$ and assume $u \rightarrow 1 \rightarrow 2$ in $\mu_1(Q)$.
%So we have $2 \rightarrow u$ in $\mu_1(Q)$ and $b_{2u}(\mu_1(Q)) > b_{u2}(Q)$.

We split into cases depending on if $\mu_1(Q)^{mut}$ has a source at $1$.

If $1$ \textbf{is a source} in $\mu_1(Q)^{mut}$, then, without loss of generality, say  $2 \rightarrow 3$. We have the following orientations of $\mu_1(Q)$:
\begin{center}
\begin{tikzpicture}[->, >={Stealth[round]}, node distance=3cm and 4cm, main/.style = {circle, draw, fill=none, minimum size=6mm, inner sep=0pt}, every edge/.append style={thick}, square/.style={regular polygon,regular polygon sides=4}]

\node[main] (1) {$\boldsymbol{1}$};
\node[main] (2) [right=2cm of 1] {$\boldsymbol{2}$};
\node[main] (3) [below=2cm of 2] {$\boldsymbol{3}$};
\node[draw, square, inner sep=.8mm] (4) [below=2cm of 1] {$\boldsymbol{u}$};

\node (Q) [below left=0.8cm of 1]{$\mu_{m_1}(Q):$}; %this causes compilation to fail??

\path 
    (2) edge (3)
    (1) edge (2)
    (1) edge (3)
    (4) edge (1)
    (2) edge (4)
    (3) edge (4);
\end{tikzpicture}
\end{center}
By \Cref{lem: rank 2 ascents} (3), mutation at either $2$ or $3$ in $Q^{(1)}_{\mathbf{M}}$ will increase $|b_{1u}|$. Mutation at $3$ doesn't change any other weights adjacent to $u$, so $3$ is a $u$-ascent. Vertex $2$ is the elbow of $Q^{(1)}_{\mathbf M}|_{23u}$, so $2$ is also a $u$-ascent.

If $1$ \textbf{is not a source} in $\mu_1(Q)^{mut}$, then we have $2 \rightarrow 3 \rightarrow 1$ in $\mu_1(Q)$. 
So we have the following orientations:
%In $\mu_{1}(Q)$ there are two subquivers which are oriented cycles: $\mu_{1}(Q)|_{12u}$ and $\mu_{m_1}(Q)|_{123}$. 

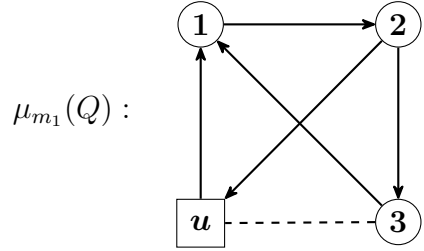
\begin{center}
\begin{tikzpicture}[->, >={Stealth[round]}, node distance=3cm and 4cm, main/.style = {circle, draw, fill=none, minimum size=6mm, inner sep=0pt}, every edge/.append style={thick}, square/.style={regular polygon,regular polygon sides=4}]

\node[main] (1) {$\boldsymbol{1}$};
\node[main] (2) [right=2cm of 1] {$\boldsymbol{2}$};
\node[main] (3) [below=2cm of 2] {$\boldsymbol{3}$};
\node[draw, square, inner sep=.8mm] (4) [below=2cm of 1] {$\boldsymbol{u}$};

\node (Q) [below left=0.8cm of 1]{$\mu_{m_1}(Q):$}; %this causes compilation to fail??

\path 
    (2) edge (3)
    (3) edge (1)
    (1) edge (2)
    (4) edge (1)
    (2) edge (4)
    (3) edge[dashed, -] (4);
\end{tikzpicture}
\end{center}
%(The dashed line between $3$ and $u$ could have either orientation, or be missing.)

Mutation at $2$ is a $u$-ascent by \Cref{lem: rank 2 ascents} (3). 
If $3$ is adjacent to $u$ in $Q$, then, regardless of the missing arrow orientation, mutation at $3$ is a $u$-ascent because $3$ is the elbow of either $\mu_{1}(Q)|_{13u}$ or $\mu_{m_1}(Q)|_{23u}$.
On the other hand, if $3$ is not adjacent to $u$ then both $u \rightarrow 1 \rightarrow 3$ and $u \leftarrow 2 \leftarrow 3$ in $Q_{\mathbf M}^{(2)} = \mu_3(\mu_1(Q))$. 
Thus regardless of the value of $m_3 \in \{1,2\}$, vertex $m_3$ is a $u$-ascent in $Q_{\mathbf M}^{(2)}$.
\end{proof}

\begin{lem}
\label{lem: cycle preserving tail}
Let $Q$ be a connected rank $3$ quiver %with one frozen vertex $u$ 
such that $Q^{\mut}$ is mutation-abundant. Let $\mathbf{M}$ be a reduced and weakly balanced mutation sequence such that $\mathbf{M}$ is cycle-preserving on $Q^{mut}$. 
Then there is some $i$ such that $\mathbf{M'} = m_im_{i+1}\cdots$ is cycle-preserving on $Q^{(i-1)}$.
\end{lem}
\begin{proof}
It suffices to check when $Q$ has just one frozen vertex $u$.
\Cref{lem:weakly balanced grows} states that eventually some mutation $m_i$ in $\mathbf{M}$ is a $u$-ascent.
By \Cref{lem: increasing sticky}, all mutations thereafter are $u$-ascents or are not adjacent to $u$.
Mutations at $u$-ascents are cycle-preserving by \Cref{lem: ascents are cycle preserving}.
If $i$ is not adjacent to $u$, then the only complete subquiver containing $i$ is $Q^{mut}$, and $i$ is cycle-preserving for $Q^{mut}$ by assumption.
Thus the mutation sequence $\mathbf{M'}=m_im_{i+1}\cdots$ is cycle-preserving.
\end{proof}

\begin{lem}
\label{lem: complete tail}
Let $Q$ be a connected rank $3$ quiver with $Q^{mut}$ mutation-abundant.
Let $\mathbf M$ be a reduced and weakly balanced mutation sequence which is cycle-preserving on $Q^{mut}$. 
Then there exists an $i$ such that for all $j>i$, $Q^{(j)}_{\mathbf M}$ is complete.
\end{lem}
\begin{proof}
It suffices to check when $Q$ has just one frozen vertex $u$.
By \Cref{lem: cycle preserving tail}, we may assume that $\mathbf M$ is cycle-preserving.
Because we only remove arrows from oriented cycles, the set of edges in the underlying undirected graph of $Q^{(j)}_{\mathbf M}$ is non-decreasing.
Note that if $i$ is cycle preserving then the orientations of $\mu_i(Q)$ are entirely determined by the orientations of $Q$. 
So we consider the possible directed graphs. 
If $Q^{mut}$ is acyclic, then  $(Q^{(6)}_{\mathbf{M}})^{mut}$ is complete or an oriented cycle (we may only perform $5$ consecutive source/sink mutations). %make a lolipop with handle u, mutate away from the handle, then once at the neighbor, then away again.
So we assume $Q^{mut}$ is an oriented cycle.
Up to global reversal of arrows and relabeling, assume that $1 \rightarrow 2$ and $u \rightarrow 1$ in $Q$. (See also \Cref{fig: cycle preserving on incomplete quivers}.)
\begin{enumerate}
\item 
If $Q$ is complete, we are done.
\item 
If $Q$ has $3 \rightarrow u$ then a cycle-preserving mutation at $1$ or $3$ results in a complete quiver. A cycle-preserving mutation at $2$ sends us to the next case (up to isomorphism).
\item 
If $Q$ has $2 \rightarrow u$ then a cycle preserving mutation at $1$ or $2$ will return us to this same case (after relabeling $1$ and $2$). A cycle preserving mutation at $3$ puts us in the previous case (up to isomorphism).
\item 
If $Q$ has a source or sink at $u$, then mutations not adjacent to $u$ return us to this case. A mutation at a neighbor of $u$ sends us to one of the other cases.
\end{enumerate}
We note that if we go from case $(3)$ to case $(2)$ then the next mutation necessarily gives us a complete quiver, case $(1)$, (the mutation back to case $(2)$ is forbidden because our mutation sequence is reduced). 
Because our mutation sequence is weakly balanced, eventually we reach a complete quiver.
\end{proof}

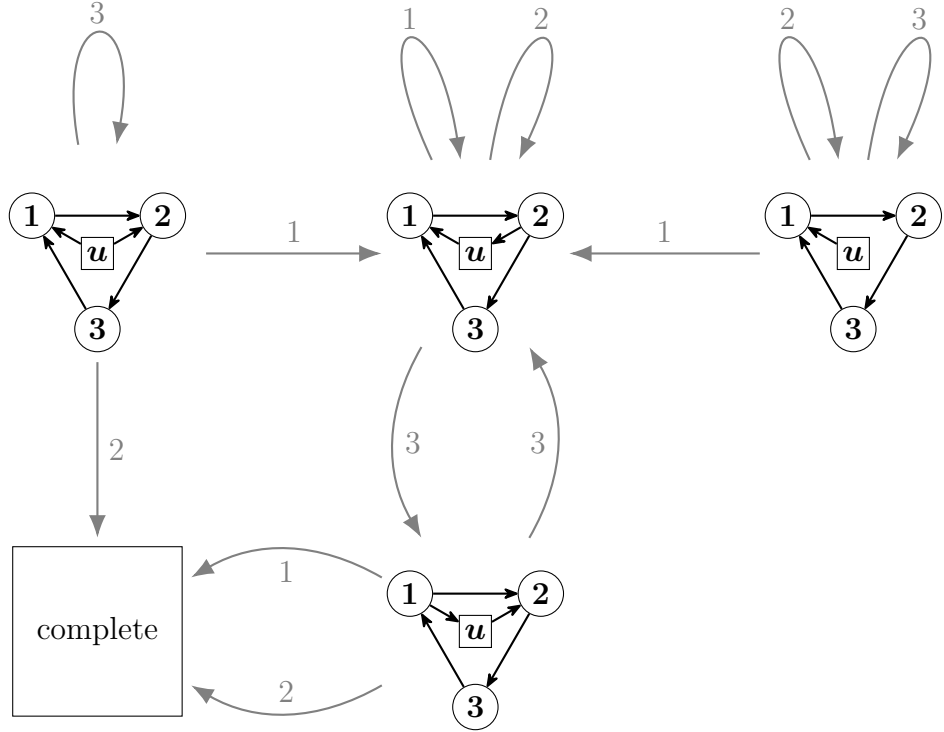
\begin{figure}
\centering
\begin{tikzpicture}[->, >={Stealth[round]}, node distance=3cm and 4cm, main/.style = {circle, draw, fill=none, minimum size=6mm, inner sep=0pt}, every edge/.append style={thick}, square/.style={regular polygon,regular polygon sides=4}]

\node[draw=none, outer sep= 1.1cm] (complete) at (0,0) {};
\begin{scope}[shift = {(complete)}] %complete
%\filldraw[black] (0,0)++(150:1cm) node[main, fill = white] (1) {\textcolor{black}{$\boldsymbol{1}$}};
%\filldraw[black] (0,0)++(30:1cm) node[main, fill = white] (2) {\textcolor{black}{$\boldsymbol{2}$}};
%\filldraw[black] (0,0)++(-90:1cm) node[main, fill = white] (3) {\textcolor{black}{$\boldsymbol{3}$}};
%\filldraw[black] (0,0) node[main, fill = white, square] (u) {\textcolor{black}{$\boldsymbol{u}$}};
\node[draw, square, inner sep=-0.3] at (0,0) {complete}; %remove the node

%\path 
%    (2) edge (3)
%    (3) edge (1)
%    (1) edge (2)
%    (u) edge[-] (1)
%    (u) edge[-] (3)
%    (2) edge[-] (u);
\end{scope}

\node[draw=none, outer sep= 1.1cm] (1nbr) at (10,5) {};
\begin{scope}[shift = {(1nbr)}] %one nbr
\filldraw[black] (0,0)++(150:1cm) node[main, fill = white] (1) {\textcolor{black}{$\boldsymbol{1}$}};
\filldraw[black] (0,0)++(30:1cm) node[main, fill = white] (2) {\textcolor{black}{$\boldsymbol{2}$}};
\filldraw[black] (0,0)++(-90:1cm) node[main, fill = white] (3) {\textcolor{black}{$\boldsymbol{3}$}};
\filldraw[black] (0,0) node[main, fill = white, square] (u) {\textcolor{black}{$\boldsymbol{u}$}};

\path 
    (2) edge (3)
    (3) edge (1)
    (1) edge (2)
    (u) edge (1);
\end{scope}

\node[draw=none, outer sep= 1.3cm] (2nbrSource) at (0,5) {};
\begin{scope}[shift = {(2nbrSource)}] %two nbr source
\filldraw[black] (0,0)++(150:1cm) node[main, fill = white] (1) {\textcolor{black}{$\boldsymbol{1}$}};
\filldraw[black] (0,0)++(30:1cm) node[main, fill = white] (2) {\textcolor{black}{$\boldsymbol{2}$}};
\filldraw[black] (0,0)++(-90:1cm) node[main, fill = white] (3) {\textcolor{black}{$\boldsymbol{3}$}};
\filldraw[black] (0,0) node[main, fill = white, square] (u) {\textcolor{black}{$\boldsymbol{u}$}};

\path 
    (2) edge (3)
    (3) edge (1)
    (1) edge (2)
    (u) edge (1)
    (u) edge (2);
\end{scope}

\node[draw=none, outer sep= 1.1cm] (elbow) at (5,0) {};
\begin{scope}[shift = {(elbow)}] %just one cycle, elbow
\filldraw[black] (0,0)++(150:1cm) node[main, fill = white] (1) {\textcolor{black}{$\boldsymbol{1}$}};
\filldraw[black] (0,0)++(30:1cm) node[main, fill = white] (2) {\textcolor{black}{$\boldsymbol{2}$}};
\filldraw[black] (0,0)++(-90:1cm) node[main, fill = white] (3) {\textcolor{black}{$\boldsymbol{3}$}};
\filldraw[black] (0,0) node[main, fill = white, square] (u) {\textcolor{black}{$\boldsymbol{u}$}};

\path 
    (2) edge (3)
    (3) edge (1)
    (1) edge (2)
    (1) edge (u)
    (u) edge (2);
\end{scope}

\node[draw=none, outer sep= 1.1cm] (2cycles) at (5,5) {};
\begin{scope}[shift = {(2cycles)}] %two cycles
\filldraw[black] (0,0)++(150:1cm) node[main, fill = white] (1) {\textcolor{black}{$\boldsymbol{1}$}};
\filldraw[black] (0,0)++(30:1cm) node[main, fill = white] (2) {\textcolor{black}{$\boldsymbol{2}$}};
\filldraw[black] (0,0)++(-90:1cm) node[main, fill = white] (3) {\textcolor{black}{$\boldsymbol{3}$}};
\filldraw[black] (0,0) node[main, fill = white, square] (u) {\textcolor{black}{$\boldsymbol{u}$}};

\path 
    (2) edge (3)
    (3) edge (1)
    (1) edge (2)
    (u) edge (1)
    (2) edge (u);
\end{scope}

\path[-{Latex[length=3mm]}, color = gray] (2cycles) edge["3", bend right] (elbow)
    (elbow) edge["3", bend right] (2cycles)
    (1nbr) edge["1"'] (2cycles)
    (2nbrSource) edge["1"] (2cycles)
    (2nbrSource) edge["2"] (complete)
    (elbow) edge["1", bend right] (complete)
    (elbow) edge["2"', bend left] (complete)
    (2cycles) edge[loop above, in=65, out=81, looseness=15] node {$2$} (2cycles)
    (2cycles) edge[loop above, in=99, out=115, looseness=15] node {$1$} (2cycles)
    (1nbr) edge[loop above, in=65, out=81, looseness=15] node {$3$} (1nbr)
    (1nbr) edge[loop above, in=99, out=115, looseness=15] node {$2$} (1nbr)
    (2nbrSource) edge[loop above, in=80, out=100, looseness=10] node {$3$} (2nbrSource);
\end{tikzpicture}
\caption{This shows the underlying directed graphs of $4$-vertex quivers with one frozen vertex $u$ up to isomorphism and arrow reversal, and how they change under cycle-preserving mutations. An arrow $Q \,\textcolor{gray}{\stackrel{i}{\rightarrow}}\, Q'$ means that if $i$ is cycle-preserving for $Q$ then $\mu_i(Q)$ has the same orientations as $Q'$ (up to isomorphism).}
\label{fig: cycle preserving on incomplete quivers}
\end{figure}

\begin{proof}[({Proof of \Cref{thm: mutation cyclic eventually sign coherent}})]
By \Cref{lem: cycle preserving tail} and \Cref{lem: complete tail}, there is some $i$ such that $Q^{(i-1)}_{\mathbf M}$ is complete and $\mathbf M' = m_im_{i+1}\cdots$ is cycle preserving.
%As $\mathbf{M}$ is cycle-preserving for $Q^{mut}$ the only potentially complete $4$-vertex subquivers include $Q^{mut}$ and one frozen vertex, 
By \Cref{lem: cycle preserving ascents avoid vortices} no mutation in $\mathbf M$ is at the apex of a vortex.
The conclusion follows from \Cref{thm: brog coherence}, since $Q^{(i)}_{\mathbf M}$ is then brog and hence becomes eventually sign-coherent.
\end{proof}

\begin{definition}
A $3$-vertex quiver is \emph{mutation-cyclic} if it is not mutation-acyclic. We will only use this notation for quivers with $3$ vertices.

If $Q^{\mut}$ is mutation-cyclic, then any mutation sequence is cycle-preserving on $Q$ and $Q^{\mut}$ is mutation-abundant.
\end{definition}

\begin{cor}
\label{cor: mutation cyclic eventually sign coherent}
Suppose $Q$ is a connected quiver with mutation-cyclic mutable part.
Then $Q$ is eventually sign-coherent on every reduced weakly balanced mutation sequence.
\end{cor}

We now consider abundant acyclic quivers with $3$ vertices. We show that:

\begin{thm}
Let $Q$ be a rank $3$ quiver such that $Q^{\mut}$ is mutation equivalent to an abundant acyclic quiver, and $\mathbf M$ a reduced and weakly balanced mutation sequence of $Q$.
Then $Q$ is eventually sign-coherent on $\mathbf M$.
\end{thm}

We have the following dichotomy:

\begin{lem}
\label{lem:acyclic eventually cycle-preserving too}
Let $Q$ be a rank $3$ quiver with no frozen vertices that is mutation-abundant, and $\mathbf M$ a reduced mutation sequence of $Q$.
Then either 
\begin{enumerate}
\item $Q^{(i)}_{\mathbf M}$ is acyclic for all $i$ sufficiently large (and so each mutation $m_i$ is at a sink/source, and $\mathbf M$ is eventually the periodic sequence $\overline{123}$), or
\item $Q^{(i)}_{\mathbf M}$ is an oriented cycle for all $i$ sufficiently large.
\end{enumerate}
In either case, $\mathbf M$ is eventually cycle-preserving.
\end{lem}

\begin{proof}
Suppose $Q^{(i-1)}_{\mathbf M}$ is acyclic but $Q^{(i)}_{\mathbf M} = \mu_{m_i}(Q^{(i-1)}_{\mathbf M})$ is not.
Then $m_i$ is the elbow in the abundant acyclic quiver $Q^{(i-1)}_{\mathbf M}$. 
So $Q^{(i)}_{\mathbf M}$ is a fork with point of return $m_i$, and the conclusion follows from \Cref{lem:tree lemma}.
\end{proof}

Given \Cref{lem:acyclic eventually cycle-preserving too} and \Cref{thm: mutation cyclic eventually sign coherent} we have the following.

\begin{cor}
\label{cor: rank 3 sign coherent}
Let $Q$ be a connected rank $3$ quiver whose mutable part is mutation-abundant, and let $\mathbf M$ be a reduced weakly balanced mutation sequence.
Then there is some $i$ such that $Q^{(j)}_{\mathbf M}$ is sign-coherent for all $j > i$.
\end{cor}

\subsection{Rank 2 Quivers}
%Note that our earlier methods apply to all wild-type rank $2$ quivers (i.e., where $b_{ij} \geq 3$), while the present methods additionally handle the Kronecker case (where $b_{ij} = 2$).

\begin{prop}\label{prop: rank 2 acyclic}
Let $Q$ be a connected quiver with two mutable vertices $1,2$ and one frozen vertex $u$ such that $|b_{12}| \geq 2$. Then for any reduced mutation sequence $\mathbf M$, there are at most four values of $\ell$ such that $Q^{(\ell)}_{\mathbf M}$ is acyclic. 
%Then $Q$ is mutation-equivalent to at most $8$ quivers $Q'$ such that either $Q' \setminus \{u\}$ or $Q' \setminus \{v\}$ is acyclic.
\end{prop}
\begin{proof}
Note that there are only two reduced mutation sequences on $Q$, namely $\overline{12}$ and $\overline{21}$. 
%We claim that for all but at most four values of $i$, the quiver $Q^{(i)}_{\mathbf M}$ is cyclic.  
Fix an acyclic quiver in the mutation class. 
If it is complete, then there are three acyclic quivers, related by two sink/source mutations.  
If it is not complete, then there are four acyclic quivers related by sink or source mutations.
If $Q_\mathbf{M}^{(\ell-1)}$ is acyclic but $Q_\mathbf{M}^{(\ell)}$ is cyclic, then the edge opposite $m_\ell$ must have weight larger than the edge opposite $m_{\ell + 1}$ in $Q_\mathbf{M}^{(\ell)}$.  A simple induction then shows that $Q_\mathbf{M}^{(k)}$ is cyclic for $k \geq \ell$. 
%Performing the same analysis on $Q|_{12v}$ shows that for at most $8$ values of $i$, one of the $3$-vertex subquivers is acyclic.
\end{proof}

\begin{thm}[cf. {\cite[Theorem 3.5]{GekhtmanNakanishi}}]
  If $Q$ is a connected mutation-infinite quiver of rank $2$,  then the asymptotic sign coherence conjecture (\Cref{conj: asymptotic sign coherence}) holds for $Q$.  %Then the mutation sequence $\overline{12}$ is eventually sign coherent.
\end{thm} 
\begin{proof}
 With out loss of generality, consider the mutation sequence $\mathbf M = \overline{12}$. Let $u_1,\dots,u_m$ be the frozen vertices of $Q$.  
%After finitely many mutations, all quivers $Q'$ appearing in the sequence must be such that both $Q' \setminus \{u\}$ or $Q' \setminus \{v\}$ are cyclic.  Thus the resulting quiver is sign-coherent.  \AB{It would be nice to do this in the language of forks.  However, I think having the Markov quiver as a subquiver poses a problem to such a proof.}
%\end{proof}

If each of the $3$-vertex subquivers $Q|_{12u_i}$ is an oriented cycle, then $Q$ is sign-coherent.
Thus it suffices to show that $Q|_{12u_i}$ will become and remain an oriented cycle after finitely many mutations.
Because $Q$ is mutation-infinite, a brief computation shows that the weight $|b_{12}|$ is at least $2$.

By \Cref{prop: rank 2 acyclic}, for a fixed $i \in [m]$, there are at most $4$ values of $\ell \in \N$ such that $Q_{\mathbf M}^{(\ell)}|_{12u_i}$ is acyclic.  Thus there are finitely many (in particular, at most $4m$) values of $i$ such that any subquiver $Q_{\mathbf M}^{(\ell)}|_{12u_i}$ is acyclic.  
\end{proof}

\iffalse
\begin{prop}
\label{prop: rank 2 eventually sign coherent}
\cS{accidental duplication of the above. Choose or mix/match}
Let $Q$ be a connected rank $2$ mutation infinite quiver with frozen vertices $u_1,u_2,\ldots,u_m$. Then both infinite reduced mutation sequences of $Q$ are eventually sign-coherent.
\end{prop}

\begin{proof}

If each of the $3$-vertex subquivers $Q|_{12u_i}$ is an oriented cycle, then $Q$ is sign-coherent. Thus it suffices to show that $Q|_{12u_i}$ will become and remain an oriented cycle after finitely many mutations. With out loss of generality, consider the mutation sequence $\mathbf M = \overline{12}$. Because $Q$ is mutation-infinite, a brief computation shows that the weight $|b_{12}|$ is at least $2$.

With out loss of generality, consider the mutation sequence $\mathbf M = \overline{12}$. 
There are three cases.
\begin{enumerate}
    \item If $Q|_{12u_i}$ is acyclic, then after at most $5$ mutations, $Q|_{12u_i}$ is an ice fork and remains so.
    \item If $Q|_{12u_i}$ is an oriented cycle and $|b_{1u_i}| \geq |b_{2u_i}|$, then $(Q|_{12u_i})^{(j)}_{\mathbf M}$ will be an oriented cycle for all $j$.
    \item if $Q|_{12u_i}$ is an oriented cycle and $|b_{1u_i}| < |b_{2u_i}|$, then mutation at $1$ reduces the number of arrows in $Q|_{12u_i}$. A mutation like this can only happen finitely many times consecutively before we arrive in one of the previous two cases.
\end{enumerate}
Therefore each of the subquivers  $Q|_{12u_i}$ is mutation infinite, and in particular is mutation equivalent to an ice fork.

\end{proof}
\fi

\subsection{Finite Ice-Forkless Part}

There are many quivers such that all but finitely many quivers in their mutation class are ice forks.  It follows quickly from work in \Cref{sec:fork eventual coherence} that these quivers satisfy the asymptotic sign coherence conjecture (\Cref{conj: asymptotic sign coherence}).  

\begin{cor}
Let $Q$ be a mutation-infinite quiver such that all but finitely many quivers in its mutation class are ice forks.  Then for any simple weakly balanced mutation sequence $\mathbf M$, $Q$ is eventually sign-coherent on $\mathbf M$.
\end{cor}
\begin{proof}
By the condition that $\mathbf M$ is simple, for sufficiently large $j$ all quivers $Q_{\mathbf M}^{(j)}$ must be ice forks.  Moreover, simplicity ensures that except for the first mutation, we do not mutate at the point of return of an ice fork.  Thus eventually $\mathbf M$ satisfies the hypotheses of \Cref{thm:Asym Sign Coherence}, hence $Q$ is eventually sign-coherent on $\mathbf M$.
\end{proof}

It is an interesting open problem to give a criterion for whether a mutation class has finite ice-forkless part.  For a given quiver, this property can often be verified directly. 

\begin{prop}
    Given a quiver with finite forkless part, freezing any vertices yields a quiver with finite ice-forkless part.  
\end{prop}
An infinite family of quivers which have finite forkless part are constructed in \cite[Example 5.13]{COQ}. 
Hence such quivers satisfy the asymptotic sign coherence conjecture.  Other families of quivers whose mutation classes have finite forkless part are constructed in \cite[Examples 10.1-10.2]{LMC} and \cite{FordyMarsh}.

\section{Asymptotic vs. eventual sign coherence}
\label{sec:asym vs eventual}
We now discuss the correspondence between Gekhtman and Nakanishi's original conjecture and our perspective.  Gekhtman and Nakanishi take the approach of fixing only the mutable part of the quiver and mutation sequence, then iterating over all choices of frozen vertices and their adjacent arrows.  Thus their conjecture is phrased asymptotically, as there is no uniform bound on the minimum number of mutations needed.  We take the approach of fixing the entire quiver (including the frozen portion) and mutation sequence, allowing us to view sign coherence as an eventual property.  We prove a precise equivalence between these setups.

%\cS{Include a remark that almost always asymptotic sign coherence is frustratingly different than almost always eventual sign coherence.} %Specifically, Warkentin's base case in 3.4 is a connected 3-vertex quiver with only one large weight. He shows 5 mutations suffice regardless of its orientation. To get a mutation sequence that converts a large subset of extensions to a fork, we need to give a mutation sequence independent of the initial orientations. ~10 Alternating mutations at the large weight suffice. %Alternatively, all a-vectors are converted to forks by some mutation sequence of finite (constant given n) length. So every interval of length n has some chance of being one of these. 
%Note that there's no fixed mutation sequence for a mutable part which sends each quiver with that mutable part and one frozen vertex to a fork. By contradiction: Fix such a sequence, perform it, add a frozen sink, mutate back.

To each integer $a \in \Z$, we naturally associate $\sign(a) \in \{+,0,-\}$ according to whether $a$ is positive, zero, or negative.  Given $\vec{a} = (a_1,\dots,a_n) \in \Z^n$ and a quiver $Q$ with no frozen vertices, let $\widehat Q$ denote the quiver obtained by adding a single frozen vertex $u$ with $a_i$ arrows from $i$ to $u$.  Given $\vec{a} \in \Z^n$, $k \in \Z_{\geq 0}$, and a mutation sequence $\mathbf M$ on a quiver $Q$, we set
$$\sigma_{\mathbf M}^{(k)}\left(\vec{a}\right) = \left(\sign\left(\widehat b_{1u}\right), \sign\left(\widehat b_{2u}\right), \dots, \sign\left(\widehat b_{nu}\right) \right)$$
where $\widehat B = (\widehat b_{ij})$ is the exchange matrix for $\widehat Q_{\mathbf M}^{(k)}$.

\begin{conj}[{\cite[Conjecture 2.3]{GekhtmanNakanishi}, Asymptotic Sign Coherence}]\label{conj: asymptotic sign coherence}
Let $Q$ be a connected quiver with no frozen vertices and $\mathbf M$ be a monotone and (weakly) balanced mutation sequence for the principal framing of $Q$.  For any pair of vector $\vec{a}, \vec{b} \in \Z^n$, there exists $T \in \N$ such that for all $k > T$, we have
$$\sigma_{\mathbf M}^{(k)}\big(\,\vec{a}\,\big) = \sigma_{\mathbf M}^{(k)}\big(\,\vec{b}\,\big)\,.$$    
\end{conj}

    In the original formulation of \Cref{conj: asymptotic sign coherence}, Gekhtman and Nakanishi instead assert that there exists a sequence of sign vectors $\sigma_{\text{reg}}^{(k)}$ such that for any $\vec{a} \in \Z^{n}$, there exists $T_{\vec{a}} \in \N$ such that $\sigma_{M}^{(k)}\big(\,\vec{a}\,\big)  = \sigma_{\text{reg}}^{(k)}$ for $k > T_{\vec{a}}$ (using the notation from \cite[Conjecture 2.3]{GekhtmanNakanishi}).  The statement above is equivalent: for one direction just set $\sigma_{\text{reg}}^{(k)} = \sigma_{\mathbf M}^{(k)}\big(\,\vec{b}\,\big) $, and for the other direction choose $T = \max(T_{\vec{a}}, T_{\vec{b}})$.
    
    Additionally, Gekhtman and Nakanishi posed their conjecture in the more general setting of \emph{skew-symmetrizable matrices} (see \cite[Section 2]{GekhtmanNakanishi} for more details), but our restatement is restricted to our setting of quivers (equivalently, \emph{skew-symmetric} matrices).  We expect our methods could be extended to the skew-symmetrizable setting; this would require generalizing Warkentin's results on forks \cite{Warkentin} to the skew-symmetrizable setting.

Note that for some fixed $Q$ and $\mathbf M$ as in \Cref{conj: asymptotic sign coherence}, there may not be a uniform bound on the parameter $T$ that works for all choices of $\vec{a}$ and $\vec{b}$.  This can be seen in \cite[Section 4]{GekhtmanNakanishi} for the Markov quiver, where the sign vector $(+,-,+)$ can persist for arbitrarily long finite periods before eventually becoming $(-,+,-)$.

\begin{rem}
Gekhtman and Nakanishi \cite{GekhtmanNakanishi} conjecture that every \emph{balanced} and \emph{monotone} mutation sequence eventually results in a sign-coherent quiver, and this property persists forever after.
There exists quivers with balanced and monotone sequences that never produce an ice fork. See Section~\ref{subsec:rank3 asymptotic}.
\end{rem}

Immediately following Conjecture~\ref{conj: asymptotic sign coherence}, Gekhtman and Nakanishi remark that it is motivated by the emergence of sign coherence when starting from particular quivers with frozen vertices.  From this perspective, it seems more natural to consider the setting where the fixed data to be a quiver including all frozen vertices and incident arrows.  We collect this phenomenon in the following equivalent conjecture.

\begin{conj}[Eventual sign coherence, equivalent to {\cite[Conjecture 2.3]{GekhtmanNakanishi}}]
\label{conj: eventual sign coherence}
Let $Q$ be a connected quiver with at least one frozen vertex and $\mathbf M$ a reduced (weakly) balanced and monotone mutation sequence. There exists $T$ such that $Q^{(j)}_{\mathbf M}$ is sign-coherent for all $j>T$.
\end{conj}

When $Q$ has just one frozen vertex $u$, the conjecture asserts that eventually every mutable vertex will be adjacent to $u$.

\begin{prop}
Asymptotic sign coherence (Conjecture~\ref{conj: asymptotic sign coherence}) is equivalent to 
Eventual sign coherence (Conjecture~\ref{conj: eventual sign coherence}).
\end{prop}

\begin{proof}
(Conjecture~\ref{conj: asymptotic sign coherence} $\implies$ Conjecture~\ref{conj: eventual sign coherence}) 
Let $Q$ have rank $n$ with exchange matrix $B = (b_{ij})$. As $\mathbf M$ is monotone for $Q$, by \cite[Theorem 4.6]{FZiv} it is still monotone for the principal framing of $Q$.
For each frozen vertex $u$, define $\vec{a}(u) = (b_{iu})$.
By Conjecture~\ref{conj: asymptotic sign coherence}, there is some $T_{\vec{a}(u) \vec{a}(v)}$ for each pair $u,v$ of frozen vertices such that $(Q|_{[n] \cup \{u,v\}})_{\mathbf M}^{(j)}$ is sign-coherent for $j > T_{\vec{a}(u) \vec{a}(v)}$.
Thus $Q$ is sign-coherent for all $j > \max_{u,v}(T_{\vec{a}(u) \vec{a}(v)})$.
% % this doesn't hold when we condition on M monotone in different ways; but there might be some way to show that green -> green implies red -> red implies monotone for principal framing = montone for all frozen additions

(Conjecture~\ref{conj: eventual sign coherence} $\implies$ Conjecture~\ref{conj: asymptotic sign coherence}) 
Consider the quiver $R$ with $R^{\mut} = Q$ and two frozen vertices $u_a,u_b$ which are connected to the mutable vertices as described by $\vec{a}, \vec{b}.$
Then $Q'^{(j)}_{\mathbf M}$ is sign-coherent if and only if $\sigma^{(j)}_{\mathbf M}\big ( \, \vec a \,\big) = \sigma^{(j)}_{\mathbf M}\big ( \, \vec b \,\big)$.
\end{proof}

We note that Theorem~\ref{thm:Asym Sign Coherence} establishes that:
for each mutation infinite quiver $Q$ with no frozen vertices and any quiver $R$ with $R^{\mut}=Q^{\mut}$, for almost all reduced and (weakly) balanced mutation sequences $\mathbf M$ we have $Q^{(j)}_{\mathbf M}$ is sign-coherent for $j$ large enough.
This is subtly different from the similar claim (which we do not prove): 
Given a mutation infinite quiver $Q$ with no frozen vertices, then for almost all reduced and (weakly) balanced mutation sequences $\mathbf M$ and any nonzero $\vec a, \vec b \in \Z^n$, we have that $\sigma^{(j)}_{\mathbf M}\big ( \, \vec a \,\big) = \sigma^{(j)}_{\mathbf M}\big ( \, \vec b \,\big)$ for $j$ sufficiently large.
(We have made it so the mutation sequences no longer depend on the choice of arrows to the frozen vertices.)
It could be that this conclusion fails to hold for each mutation sequence because of just a small collection of pairs $\vec a, \vec b$.

\bibliographystyle{plain}
\bibliography{bibliography.bib}

\end{document}